\documentclass[11pt]{article}

\usepackage{amsmath}
\usepackage{amssymb, yhmath, ulem}
\usepackage[ruled,noend,linesnumbered]{algorithm2e}
\usepackage{float}

\usepackage{enumerate}
\usepackage{listings}

\usepackage[pdftex]{graphicx}
\usepackage{epstopdf}
\usepackage{bm}
\usepackage{authblk}

\usepackage{color}
\definecolor{dkgreen}{rgb}{0,0.6,0}
\definecolor{gray}{rgb}{0.5,0.5,0.5}

\usepackage[bookmarks,pagebackref,pdfpagelabels=true]{hyperref}

\hypersetup{
    colorlinks=true,        
    linkcolor=red,          
    citecolor=green,        
    filecolor=magenta,      
    urlcolor=dkgreen           
}
\usepackage{kbordermatrix}

\usepackage[c]{documentation}

\newcommand{\CURID}{{\sf{CUR-ID}}}
\newcommand{\CUR}{{\sf{CUR}}}

\newcommand{\ID}{{\sf{ID}}}
\newcommand{\QR}{{\sf{QR}}}
\newcommand{\QB}{{\sf{QB}}}
\newcommand{\SVD}{{\sf{SVD}}}

\newtheorem{theorem}{Theorem}[section]

\newtheorem{lemma}[theorem]{Lemma}

\newenvironment{proof}[1][Proof]{\begin{trivlist}\item[\hskip \labelsep {\bfseries #1.}]}{$\Box$\end{trivlist}}

\newcommand{\range}{\operatorname{\mathcal{R}}}

\newcommand{\orth}{\texttt{orth}}

\newcommand{\qr}{\texttt{qr}}
\newcommand{\svd}{\texttt{svd}}
\numberwithin{equation}{section}

\newcommand{\mtx}[1]{\bm{\mathsf{#1}}}
\newcommand{\vct}[1]{\bm{\mathsf{#1}}}

\title{RSVDPACK: An implementation of randomized algorithms for
computing the singular value, interpolative, and CUR decompositions
of matrices on multi-core and GPU architectures}

\author{
{Sergey Voronin and Per-Gunnar Martinsson}
}
\date{\today}

\begin{document}

\maketitle

\lstset{language=Matlab,
   keywords={break,case,catch,continue,else,elseif,end,for,function,
      global,if,otherwise,persistent,return,switch,try,while},
   basicstyle=\ttfamily,
   keywordstyle=\color{blue},
   commentstyle=\color{red},
   stringstyle=\color{dkgreen},
   numbers=left,
   numberstyle=\tiny\color{gray},
   stepnumber=1,
   numbersep=10pt,
   backgroundcolor=\color{white},
   tabsize=4,
   showspaces=false,
   showstringspaces=false}

\begin{abstract}
RSVDPACK is a library of functions for computing low rank approximations of matrices.
The library includes functions for computing standard (partial) factorizations such as the Singular Value
Decomposition (SVD), and also so called ``structure preserving'' factorizations such as the
Interpolative Decomposition (ID) and the CUR decomposition. The ID and CUR factorizations
pick subsets of the rows/columns of a matrix to use as bases for its row/column space.
Such factorizations preserve properties of the matrix such as sparsity or non-negativity,
are helpful in data interpretation, and require in certain contexts less memory than a partial SVD.
The package implements highly efficient computational algorithms based on randomized sampling,
as described and analyzed in
\textit{N. Halko, P.G. Martinsson, J. Tropp, ``Finding structure with randomness:
Probabilistic algorithms for constructing approximate matrix decompositions,''
SIAM Review, \textbf{53}(2), 2011,} and subsequent papers. This manuscript presents some modifications
to the basic algorithms that improve performance and ease of use.
The library is written in C and supports both multi-core CPU and GPU architectures.
\end{abstract}

\section{Introduction}
This manuscript describes a collection of functions for computing low-rank approximations
to matrices. In other words, given an $m\times n$ matrix $\mtx{A}$ stored in RAM, we seek to
compute an approximation $\mtx{A}_{\rm approx}$ of rank $k < \min(m,n)$, represented in
factored form. We
consider the case where $\mtx{A}_{\rm approx}$ is an approximate singular value
decomposition (SVD), and also the case where $\mtx{A}_{\rm approx}$ is represented
in a so called ``structure preserving'' factorization such as the CUR or interpolative
decompositions, see \cite{2005_martinsson_skel,2014arXiv1412.8447V,2009_mahoney_CUR}.
The problems addressed arise frequently
in scientific computing, data analysis, statistics, and many other areas.

Among the different factorizations, the partial singular value decomposition
is known to be optimal in the sense that for any given rank, it results in a minimal error
$\|\mtx{A} - \mtx{A}_{\rm approx}\|$, as measured in either the $\ell_{2}$-operator norm, or the
Frobenius norm. The interpolative and CUR decompositions provide for larger than
minimal error at any given rank, but preserve certain useful properties
such as sparsity and non-negativity.

The algorithms used are based on randomized sampling, and are highly computationally
efficient. In particular, the developed software aims at reduced communication cost and
good scalability on multi-core/processor systems.

The SVD algorithms used here were originally published in
\cite{2006_martinsson_random1_orig}, were later extended in
\cite{2007_martinsson_PNAS} and analyzed and surveyed in
\cite{2011_martinsson_randomsurvey}. For other decompositions,
we have made use of more recent results from \cite{2015arXiv150307157M} 
and \cite{2014arXiv1412.8447V}, which were inspired by \cite{2009_mahoney_CUR}. 
Related work is reported in \cite{boutsidis2014optimal,sorensen2016deim}.
In our development, we made some
modifications to previously published versions and implemented what we believe to be
the most computationally efficient and practical algorithmic variants for use with
applications.

To introduce the idea of randomized algorithms for computing low rank
approximations to matrices, we show in Figure \ref{fig:RSVD} a basic
randomized algorithm called RSVD for computing an approximation to the
dominant $k$ modes in a singular value decomposition (SVD) of a given
matrix $\mtx{A}$. The
algorithm shown is intended for use in the case where the rank $k$ is
much smaller than the matrix dimensions, $k \ll \min(m,n)$. In this
environment, RSVD tends to execute very fast since all interactions
with the large matrix $\mtx{A}$ happen only through
the matrix-matrix multiplications on lines (2) and (4). The matrix-matrix
factorization is a communication efficient algorithm for which
highly optimized software is available on most computing platforms.
In particular, the matrix-matrix multiplication executes very fast
on modern multi-core CPUs and massively multi-core GPUs. All operations
in the algorithm that are not matrix-matrix multiplications involve
small matrices that have either roughly $k$ rows or roughly $k$ columns (to be
precise, they have $k+p$ rows or columns, where $p$ is a small ``over-sampling
parameter'' that we typically set to $5$ or $10$).

In this paper, we describe efficient implementations of the algorithm
shown in Figure \ref{fig:RSVD}, as well as some algorithms with additional
features that extend the range of problems that can be handled.
These algorithms include functions that achieve high computational efficiency
in cases where the numerical rank of the matrix is not known in advance, and
instead must be determined as part of the computation (given a requested
tolerance). They also include variations of the basic algorithm that incur
slightly higher computational costs, but in return produce close to optimally
accurate results even for matrices with ``noisy'' entries such as, e.g., measured
statistical data. (To be precise, these modified algorithms are designed
for matrices whose singular values decay slowly.) The high computational
performance attained by these algorithms can be largely attributed to one
recurring idea:

\begin{quote}
\textit{\textbf{Key idea:} Use randomization to cast as much
of the computation as possible in terms of highly efficient
matrix-matrix multiplications.}
\end{quote}

The algorithms we discuss can readily be implemented directly in Matlab, which for many
users may be sufficient. We remark also that in recent time, other software for randomized
decompositions has been developed, for example, in the form of routines for R
\cite{Erichson16}; as well as codes in Fortran \cite{martinsson2008id} and 
Python \cite{fbpca}. 
The C based routines in RSVDPACK are meant to be used for larger sized applications 
where computational efficiency and parallel scalability are important and include 
optimizations for multi-core processors and GPUs.
The codes also incorporate some of the latest randomized methods for
\SVD, \ID, and \CUR~computations refined by the authors and provide the possibility
to use an input tolerance parameter instead of a fixed rank.


\begin{figure}
\fbox{\begin{minipage}{\textwidth}
(1) Draw an $n\times (k+p)$ Gaussian random matrix $\mtx{G}$.\hfill
      \texttt{\color{red}G = randn(n,k+p)}

(2) Form the $m\times (k+p)$ sample matrix $\mtx{Y} = \mtx{A}\,\mtx{G}$.\hfill
      \texttt{\color{red}Y = A * G}

(3) Form an $m\times (k+p)$ orthonormal matrix $\mtx{Q}$ such that $\mtx{Y} = \mtx{Q}\,\mtx{R}$.\hfill
      \texttt{\color{red}[Q, R] = qr(Y,0)}

(4) Form the $(k+p)\times n$ matrix $\mtx{B} = \mtx{Q}^{*}\,\mtx{A}$.\hfill
      \texttt{\color{red}B = Q' * A}

(5) Compute the SVD of the small matrix $\mtx{B}$: $\mtx{B} = \hat{\mtx{U}}\,\mtx{D}\,\mtx{V}^{*}$.\hfill
      \texttt{\color{red}[Uhat, D, V] = svd(B,'econ')}

(6) Form the matrix $\mtx{U} = \mtx{Q}\,\hat{\mtx{U}}$.\hfill
      \texttt{\color{red}U = Q * Uhat}

(7) Truncate the trailing $p$ terms. \hfill
      \texttt{\color{red}U = U(:,1:k); V = V(:,1:k); D = D(1:k,1:k)}
\end{minipage}}
\caption{A randomized algorithm for computing an approximate singular value decomposition
of a given matrix. The inputs are an $m\times n$ matrix $\mtx{A}$, a target rank $k$, and
an over-sampling parameter $p$ (the choice $p=5$ is often very good). The outputs are
orthonormal matrices $\mtx{U}$ and $\mtx{V}$ of sizes $m\times k$ and $n\times k$, respectively,
and a $k\times k$ diagonal matrix $\mtx{D}$ such that $\mtx{A} \approx \mtx{U}\mtx{D}\mtx{V}^{*}$.}
\label{fig:RSVD}
\end{figure}

The manuscript is organized as follows:
Section \ref{sec:matrix_decompositions} lists known
facts about matrix factorizations and the low-rank approximation problem that we need. Some
of these facts are standard results, and some are perhaps less well known, in particular facts
regarding the ``structure preserving'' factorizations.
Section \ref{sec:randomized_algorithms} reviews how randomized algorithms can be used to compute
low-rank approximations to matrices, and also includes some extensions and modifications that
have not previously been published.
Section \ref{sec:software} describes the functionality of the RSVDPACK software.
Section \ref{sec:performance} shows the results of numerical experiments that illustrate
the speed and accuracy of our software.
Section \ref{sec:conc} summarizes our key findings and discusses future work.
Section \ref{sec:availability} describes the license terms and availability of the software.

\section{Matrix decompositions}
\label{sec:matrix_decompositions}

This section introduces our notation, and describes the full and low rank decompositions which we will use.
We describe the singular value decomposition (SVD), the column pivoted QR decomposition, the one and two sided
interpolative decompositions (IDs),  and the \CUR~decomposition.
In terms of approximation error for the rank $k$ decompositions,  the truncated \SVD~is best, followed by
the \QR~and \ID~decompositions (with identical errors) and then by the \CUR.
In terms of memory requirements for dense matrices, the two \ID~decompositions
of $\mtx{A}$ require the least space, followed by the \SVD~and the \CUR.
However, if $\mtx{A}$ is a sparse matrix and a sparse storage format is used
for the factor matrices, the \ID~and \CUR~decompositions can be stored more efficiently
than the \SVD. In the sparse case, the \CUR~storage requirement will in many cases be minimal
amongst all the factorizations. The details of the factorizations appear in the subsections below,
while the pseudocode for the algorithms to compute the one sided \ID, two sided \ID, and
\CUR~factorizations appear in Appendix \ref{app:pseudo}.

For further details, the material on the SVD is covered in most standard textbooks, e.g., \cite{golub}.
The ID and CUR decompositions are described in further detail in, e.g.,
\cite{2005_martinsson_skel,2015arXiv150307157M,2014arXiv1412.8447V,2009_mahoney_CUR,boutsidis2014optimal,sorensen2016deim}.

\subsection{Notation}
In what follows, we let $\mtx{A}$ be a matrix with real entries.
The extension to the complex case is straight-forward in principle, but our code does not
yet have this capability implemented. The transpose of a matrix $\mtx{A}$ is denoted
$\mtx{A}^{*}$ to simplify the extension to complex matrices.
The norms $\|\cdot\|_{F}$ and $\|\cdot\|_2$ refer to the Frobenius and the spectral
(operator $\ell_2$) matrix norms, respectively. In relations where either matrix norm can
be used, we write $\|\cdot\|$. For vectors, $\|\cdot\|$ refers to the usual 
Euclidean norm. By $\range(\mtx{A})$ we refer to the set which is the range or column space 
of matrix $\mtx{A}$.
We say that a matrix is orthonormal (ON) if its columns form an orthonormal set.
We use the notation \orth~to refer to an unpivoted QR factorization. In other
words, given a matrix $\mtx{A}$ of size $m\times n$ with $m \geq n$, the matrix
$\mtx{Q} = \orth(\mtx{A})$ is an $m\times n$ ON matrix whose columns form
an orthonormal basis for the columns of $\mtx{A}$.
(Using Matlab notation, the operation \orth~can be implemented via compact \QR~factorization using the syntax
$[\mtx{Q},\sim] = qr(\mtx{A},0)$; observe that this is closely related to, but
not identical to, the native function \textit{orth} in Matlab.)
We use Matlab style indexing to refer to matrix row or column extraction.
Thus, $\mtx{V}(1:a,1:b)$, refers to a submatrix formed by extracting the first $a$ rows
and $b$ columns of $\mtx{V}$.
By $J_r$ and $J_c$ we denote index (integer) vectors of row and column numbers of $\mtx{A}$,
corresponding to some particular rearrangement.
We let $\mathbb{N}(0,1)$ denote a normalized Gaussian probability distribution,
and use the term \textit{GIID matrix} to refer to a matrix whose entries are drawn
independently from $\mathbb{N}(0,1)$.
Using Matlab notation, an $m\times n$ GIID matrix is generated via $\mtx{R} = randn(m,n)$). The expectation of a random variable is denoted $\mathrm{E}[\dots]$ and the variance by $\mathrm{Var}[\dots]$.

\subsection{The singular value decomposition}
Let $\mtx{A}$ be an $m\times n$ matrix with real entries.
Setting $r = \min(m,n)$, every such matrix admits a so
called ``economic singular value decomposition (SVD)'' of the form
\begin{equation}
\label{eq:svdofA}
\begin{array}{ccccccccccc}
\mtx{A} &=& \mtx{U} & \mtx{\Sigma} & \mtx{V}^{*},\\
m\times n && m\times r & r\times r & r\times n
\end{array}
\end{equation}
where $\mtx{U}$ and $\mtx{V}$ are orthonormal matrices and $\mtx{\Sigma}$ is a diagonal matrix.
The columns $(\vct{u}_{j})_{j=1}^{r}$ and $(\vct{v}_{j})_{j=1}^{r}$ of $\mtx{U}$ and $\mtx{V}$ are called
the left and right singular vectors of $\mtx{A}$, respectively, and the diagonal entries
$(\sigma_{j})_{j=1}^{r}$ of $\mtx{\Sigma}$ are the singular values of $\mtx{A}$. The singular
values of $\mtx{A}$ are ordered so that
$\sigma_{1} \geq \sigma_{2} \geq \cdots \geq \sigma_{r} \geq 0$.
In other words,
$$
\mtx{U} = \bigl[\vct{u}_{1}\ \vct{u}_{2}\ \cdots\ \vct{u}_{r}\bigr],
\qquad
\mtx{V} = \bigl[\vct{v}_{1}\ \vct{v}_{2}\ \cdots\ \vct{v}_{r}\bigr],
\qquad\mbox{and}\qquad
\mtx{\Sigma} = \mbox{diag}(\sigma_{1},\,\sigma_{2},\,\dots,\,\sigma_{r}).
$$
The factorization (\ref{eq:svdofA}) can be viewed as expressing $\mtx{A}$ as a sum of $p$
rank-one matrices
$\mtx{A} = \sum_{j=1}^{r}\sigma_{j}\,\vct{u}_{j}\,\vct{v}_{j}^{*}$.
In the setting of this article, we are primarily interested in the case where the singular values
$\sigma_{j}$ decay relatively rapidly to zero, meaning that the sum
converges rapidly. In this case, it is often helpful to approximate $\mtx{A}$ using an
approximation $\mtx{A}_{k} \approx \mtx{A}$ defined by the truncated sum
\begin{equation}
\label{eq:svdofAtrunc}
\mtx{A}_{k} = \sum_{j=1}^{k}\sigma_{j}\,\vct{u}_{j}\,\vct{v}_{j}^{*} = \mtx{U}_{k}\,\mtx{\Sigma}_{k}\,\mtx{V}_{k}^{*},
\end{equation}
where $k$ is a number less than $r$, and
$$
\mtx{U}_{k} = \bigl[\vct{u}_{1}\ \vct{u}_{2}\ \cdots\ \vct{u}_{k}\bigr],
\qquad
\mtx{V}_{k} = \bigl[\vct{v}_{1}\ \vct{v}_{2}\ \cdots\ \vct{v}_{k}\bigr],
\qquad\mbox{and}\qquad
\mtx{\Sigma}_{k} = \mbox{diag}(\sigma_{1},\,\sigma_{2},\,\dots,\,\sigma_{k}).
$$
It is well known that the truncated SVD $\mtx{A}_{k}$ is the most accurate of all
rank-$k$ approximations to $\mtx{A}$, in the following sense \cite{eckart1936approximation}:
\begin{theorem}[Eckart-Young]
\label{thm:eckartyoung}
Let $\mtx{A}$ be an $m\times n$ matrix with singular value decomposition (\ref{eq:svdofA}).
Then for any $k$ such that $1 \leq k \leq r$, the truncated SVD $\mtx{A}_{k}$,
as defined by (\ref{eq:svdofAtrunc}) is the optimal approximation to $\mtx{A}$ in the sense
that
$$
\|\mtx{A} - \mtx{A}_{k}\| = \inf\{\|\mtx{A} - \mtx{B}\|\,\colon\,\mtx{B}\mbox{ has rank }k\},
$$
where $\|\cdot\|$ is either the $\ell^{2}$-operator norm or the Frobenius norm.
The minima are given by,
\begin{equation}
\|\mtx{A} - \mtx{A}_{k}\|_2 = \sigma_{k+1},
\qquad\mbox{and}\qquad
\|\mtx{A} - \mtx{A}_{k}\|_F = \left(\sum_{j=k+1}^{r}\sigma_{j}^{2}\right)^{1/2}.
\end{equation}
\end{theorem}

\subsection{The column pivoted QR factorization and low rank approximation}
\label{sec:QRandonesideID}

Let $\mtx{A}$ be an $m\times n$ matrix with real entries as before
and set $r = \min(m,n)$. The column pivoted QR-factorization (CPQR) of $\mtx{A}$
takes the form
\begin{equation}
\label{eq:QR}
\begin{array}{ccccccc}
\mtx{A} & \mtx{P} &=& \mtx{Q} & \mtx{S}.\\
m\times n & n\times n && m\times r & r\times n
\end{array}
\end{equation}
where $\mtx{P}$ is a permutation matrix, $\mtx{Q}$ has orthonormal columns,
and $\mtx{S}$ is upper triangular.
(The upper triangular factor is more commonly written $\mtx{R}$ but we use $\mtx{S}$
to avoid confusion with the factors in the CUR decomposition.)

The QR factorization is commonly computed via iterative algorithms such as Gram-Schmidt
or Householder QR \cite[Sec.~5.2]{golub}, which proceed via a sequence of rank-1
updates to the matrix. When column pivoting is used, the process can be halted
after $k$ steps to produce a rank-$k$ approximation $\mtx{A}_{\rm approx}$ to $\mtx{A}$.
To illustrate, suppose that we have completed $k$ steps of the QR-factorization
process, and partition the resulting $\mtx{Q}$ and $\mtx{S}$ to split off the first $k$ columns and rows:
$$
\mtx{Q} = \kbordermatrix{& k & r-k\\ m &\mtx{Q}_1&\mtx{Q}_2},
\quad\mbox{and}\quad
\mtx{S} = \kbordermatrix{& k & n-k\\ k &\mtx{S}_{11} & \mtx{S}_{12}\\ r-k &\mtx{0} &\mtx{S}_{22}}.
$$
We can write (\ref{eq:QR}) as
\begin{equation}
\label{eqn:qrtrunc}
\mtx{A}
=
\bigl[\mtx{Q}_{1}\ \mtx{Q}_{2}\bigr]
\left[\begin{array}{cc}
\mtx{S}_{11} & \mtx{S}_{12} \\
\mtx{0}      & \mtx{S}_{22}
\end{array}\right]\mtx{P}^{*}
=
\underbrace{\mtx{Q}_1 \bigl[\mtx{S}_{11}\ \mtx{S}_{12}\bigr]\mtx{P}^{*}}_{=:\mtx{A}_{\rm approx}} +
\underbrace{\mtx{Q_2} \bigl[\mtx{0}\ \mtx{S}_{22}\bigr]\mtx{P}^{*}}_{\mbox{``remainder term''}}.
\end{equation}
The approximation error is now given by the following simple relation
$$
\|\mtx{A} - \mtx{A}_{\rm approx}\| =
\|\mtx{Q_2} \bigl[\mtx{0}\ \mtx{S}_{22}\bigr]\mtx{P}^{*}\| =
\|\bigl[\mtx{0}\ \mtx{S}_{22}\bigr]\| =
\|\mtx{S}_{22}\|.
$$
Computing a rank $k$ approximation via a partial QR factorization is typically much
faster than computing a partial singular value decomposition. The price one pays is
that the approximation error gets larger. In situations where the singular values of
$\mtx{A}$ exhibit robust decay, the sub-optimality is typically very modest
\cite{gu1996,2005_martinsson_skel}, but for certain
rare matrices, substantial sub-optimality can result \cite{1966_kahan_NLA}.

\subsection{The one-sided Interpolative Decomposition (ID)}
\label{sec:singleID}

The one-sided interpolative decomposition can be obtained by a slight amount of
post-processing of a partial CPQR. As a starting point, let us consider the situation
(\ref{eqn:qrtrunc}) that we find ourselves in after $k$ steps of the QR factorization process.
For this discussion, it is convenient to represent the permutation matrix $\mtx{P}$ using an
index vector $J_{\rm c} \in \mathbb{Z}_{+}^{n}$, where
\begin{equation}
\label{eq:Jc_index_vec}
\begin{array}{cccc}
{J}_c &=& [J_{\rm skel},\,&J_{\rm res}]\\
1\times n && 1\times k & 1\times (n-k)
\end{array},
\end{equation}
so that $\mtx{P} = \mtx{I}(:,J_{\rm c})$, where $\mtx{I}$
is the $n\times n$ identity matrix. Then $\mtx{A}\mtx{P} = \mtx{A}(:,J_{\rm c})$.
We can now write the rank $k$ approximant $\mtx{A}_{\rm approx}$ that was defined in
(\ref{eqn:qrtrunc}) as
\begin{equation}
\label{eq:lily1}
\mtx{A}_{\rm approx} =
\mtx{Q}_1 \begin{bmatrix}\mtx{S}_{11} & \mtx{S}_{12} \end{bmatrix}\mtx{P}^{*}
=
\mtx{Q}_1 \mtx{S}_{11} [\mtx{I}_k \quad \mtx{S}_{11}^{-1} \mtx{S}_{12}]\mtx{P}^{*}
=
\mtx{Q}_1 \mtx{S}_{11} [\mtx{I}_k \quad \mtx{T}_l]\mtx{P}^{*},
\end{equation}
where $\mtx{T}_l$ is the $k\times (n-k)$ matrix resulting from solving the linear system
$$
\mtx{S}_{11} \mtx{T}_l = \mtx{S}_{12}.
$$
Observe that $\mtx{S}_{11}$ necessarily has rank $k$ and is in consequence invertible.
(If the rank of $\mtx{S}_{11}$ would be less than $k$, then the exact rank of $\mtx{A}$
would also be less than $k$ and we would have halted the QR factorization earlier.)

From \eqref{eqn:qrtrunc} it follows that:
\begin{eqnarray}
\nonumber
\mtx{A}(:, J_c)
\ &=& \
\mtx{Q}_1 \begin{bmatrix} \mtx{S}_{11} & \mtx{S}_{12} \end{bmatrix}+
\mtx{Q}_2 \begin{bmatrix} \mtx{0} & \mtx{S}_{22} \end{bmatrix} \\
\label{eq:AJc_relation}
\ &=& \
\kbordermatrix{&
k & n-k\\
m & \mtx{Q}_1 \mtx{S}_{11} & \mtx{Q}_1 \mtx{S}_{12} + \mtx{Q}_2 \mtx{S}_{22}
}.
\end{eqnarray}
Now observe from \eqref{eq:AJc_relation} that the matrix $\mtx{Q}_{1}\mtx{S}_{11}$ consists
simply of the $k$ pivot columns that are listed first in $J_{\rm c}$.
We define this quantity as the $m\times k$ matrix
\begin{equation}
\label{eq:defC}
\mtx{C} := \mtx{A}(:,J_{\rm c}(1:k)) = \mtx{Q}_{1}\mtx{S}_{11}.
\end{equation}
Moreover, we define a \textit{column interpolation matrix} $\mtx{V}$ as the $n\times k$ matrix
\begin{equation}
\label{eq:defV}
\mtx{V} := \mtx{P}\left[\begin{array}{c}\mtx{I}_k \\ \mtx{T}_l^{*}\end{array}\right].
\end{equation}
Inserting (\ref{eq:defC}) and (\ref{eq:defV}) into (\ref{eq:lily1}), we find the expression
\begin{equation}
\label{eqn:column_approx_id2}
\mtx{A}_{\rm approx} = \mtx{C} \mtx{V}^{*}.
\end{equation}
Equation (\ref{eqn:column_approx_id2}) is known as a \textit{column \ID~of rank $k$} of $\mtx{A}$.
Heuristically, the column \ID~identifies a subset of the columns of $\mtx{A}$ (the $k$ columns
identified in $J_{\rm c}(1:k)$) that serve as an approximate basis for the column space of the
matrix.

The approximation error $\mtx{A} - \mtx{C}\mtx{V}^{*}$ is identical to the error in the
partial column pivoted QR factorization, since $\mtx{C}\mtx{V}^{*} = \mtx{A}_{\rm approx}$,
where $\mtx{A}_{\rm approx}$ is the approximant defined by (\ref{eqn:qrtrunc}).
Consequently,
\begin{equation}
\label{eq:id_error}
\|\mtx{A} - \mtx{C}\mtx{V}^{*}\| = \|\mtx{A} - \mtx{A}_{\rm approx}\| = \|\mtx{S}_{22}\|.
\end{equation}

Just as a \textit{column \ID} can be derived by orthonormalizing the columns of $\mtx{A}$ via a
\QR~factorization, we can also derive a \textit{row \ID} by orthonormalizing the rows of $\mtx{A}$.
Performing a $k$-step \QR~factorization of $\mtx{A}^{*}$, we end up with an approximate
factorization
\begin{equation}
\begin{array}{cccc}
\mtx{A} &\approx& \mtx{W} &\mtx{R},\\
m\times n && m\times k & k\times n
\end{array}
\end{equation}
where $\mtx{R}$ is a $k\times n$ matrix that consists of $k$ of the rows of $\mtx{A}$. To be precise,
$$
\mtx{R} = \mtx{A}(J_{\rm r}(1:k),:),
$$
where $J_{\rm r}$ is the permutation vector resulting from the QR factorization of $\mtx{A}^{*}$.

\subsection{Two sided ID and CUR Decompositions}
\label{subsec:twosidedidandcur}

The matrix factorizations described in Section \ref{sec:singleID} use either a subset of the
columns as a basis for the column space, or a subset of the rows as a basis for the row space.
Next, we will describe the \textit{two sided ID} and the CUR decompositions which select subsets
of both the columns and the rows, to serve as bases for both the column and the row spaces.

To derive the two sided \ID, we start by constructing a column-\ID~so that we have the
approximation (\ref{eqn:column_approx_id2}). Next, we execute a row-\ID~on the tall thin
matrix $\mtx{C}$, to obtain a factorization
\begin{equation}
\label{eq:factorC}
\mtx{C} = \mtx{W}\mtx{C}(J_{\rm r}(1:k),:).
\end{equation}
Observe that the factorization (\ref{eq:factorC}) is exact since the rank of $\mtx{C}$ is at most $k$.
Inserting (\ref{eq:factorC}) into (\ref{eqn:column_approx_id2}), and observing that
$$
\mtx{C}(J_{\rm r}(1:k),:) = \mtx{A}(J_{\rm r}(1:k),J_{\rm c}(1:k)),
$$
we obtain the \textit{two-sided ID}
\begin{equation}
\label{eq:doubleID}
\begin{array}{cccccc}
\mtx{A} \approx & \mtx{A}_{\rm approx} &=& \mtx{W} & \mtx{A}(J_{\rm r}(1:k),J_{\rm c}(1:k)) & \mtx{V}^{*},\\
&m\times n && m\times k & k\times k & k\times n
\end{array}
\end{equation}
Observe that the rank-$k$ approximation $\mtx{A}_{\rm approx}$ remains identical to the
matrix defined in (\ref{eqn:qrtrunc}), which means that the two-sided \ID~incurs exactly
the same error as the column \ID, and the truncated QR decomposition.

The popular \CUR~decomposition takes the form
\begin{equation}
\label{eq:defCUR}
\begin{array}{ccccccc}
\mtx{A} &\approx& \mtx{C} &\mtx{U} &\mtx{R},\\
m\times n && m\times k & k\times k & k\times n
\end{array}
\end{equation}
where the matrices $\mtx{C}$ and $\mtx{R}$ consists of $k$ columns and rows
of $\mtx{A}$, respectively, just as in Section \ref{sec:singleID}. The \CUR~decomposition
can be obtained from the two-sided \ID. As a first step, we use the index vector
$J_{\rm r}$ and $J_{\rm c}$ in (\ref{eq:doubleID}) to define
\begin{equation*}
\mtx{C} = \mtx{A}(:,J_{\rm c}(1:k)),
\qquad \mbox{and} \qquad
\mtx{R} = \mtx{A}(J_{\rm r}(1:k),:).
\end{equation*}
Then observe that since the matrix $\mtx{C}$ is the same in (\ref{eq:defCUR}) as it is in
(\ref{eqn:column_approx_id2}), we can convert the approximation (\ref{eqn:column_approx_id2})
into a \CUR~decomposition if we can determine a $k\times k$ matrix $\mtx{U}$ such that
\begin{equation}
\label{eq:CUR_ID_Urecovery1}
\mtx{U} \mtx{R} = \mtx{V}^{*},
\end{equation}
The system (\ref{eq:CUR_ID_Urecovery1}) is overdetermined, and solving it typically
incurs an additional error. Consequently, the approximation error in (\ref{eq:defCUR}) is
typically larger than the approximation error beyond the error incurred in the original
QR factorization.

\subsection{Specific rank and tolerance based decompositions}

Notice that each of the discussed decompositions (\SVD, \ID, and \CUR) can be
computed either to a certain fixed rank $k$ or to a tolerance $\textrm{TOL}$.
For each decomposition, the fixed rank $k$ refers to the size of the
product matrices in each decomposition. We can also use a tolerance to
determine the decomposition size. In the case of the low rank \SVD, given
$\textrm{TOL}$, we can compute $k$ such that $\sigma_{k+1} \leq \textrm{TOL}$.
The one and two sided \ID~decompositions are based on the pivoted \QR~factorization
and the \CUR~decomposition is based in turn on the two sided \ID. For these
decompositions, given a parameter $\textrm{TOL}$, we can perform a sufficient number of
steps in the \QR~factorization to obtain $\|\mtx{S}_{22}\| \leq \textrm{TOL}$.

\section{Randomized Algorithms}
\label{sec:randomized_algorithms}

The classical algorithms for the factorizations discussed in
Section \ref{sec:matrix_decompositions} may be too costly for matrices with a large
memory footprint. However, the algorithms to obtain all the factorizations
we have discussed: the low rank \SVD, the \ID, and \CUR~factorizations, can be substantially
accelerated by means of randomized sampling, with relatively
small accuracy tradeoffs \cite{2011_martinsson_randomsurvey}.

The factors $\mtx{U}_k$, $\mtx{\Sigma}_k$, and $\mtx{V}_k$ in the partial \SVD~of a matrix
$\mtx{A}$, cf.~(\ref{eq:svdofAtrunc}), can be computed by constructing the
\textit{full} \SVD~of $\mtx{A}$ using standard software routines such as, e.g.,
those available in a LAPACK implementation, and truncating. However, this is quite costly,
with an asymptotic cost of $O(mnr)$ where $r = \min(m,n)$.
In contrast, all techniques presented here have an asymptotic cost of $O(mnk)$,
which represents a substantial savings when $k \ll \min(m,n)$.
Notice that both the two sided \ID~and \CUR~factorizations can be computed
efficiently once a single sided \ID~has been computed, so we concern our discussion and
analysis on using randomized sampling for computing the rank $k$ \SVD~and the
rank $k$ \ID~decompositions.

The idea behind randomized algorithms for constructing low rank approximations
to a matrix is to apply the desired factorization to a smaller matrix, derived
from the original matrix. We first discuss the use of randomization for constructing
an approximate low rank \SVD, which we elaborate more on in
\ref{sec:randomized_algorithms_for_low_rank_svd}. Given $\mtx{A} \in \mathbb{R}^{m \times n}$,
we can construct samples of the column space of $\mtx{A}$ via the
computation $\mtx{Y} = \mtx{A} \mtx{\Omega}$ with $\mtx{\Omega}$ an $n \times (k+p)$
Gaussian random matrix so that $\mtx{Y}$ is of size $m \times (k+p)$ with $k$ the desired
rank and $p$ a small oversampling parameter, the use of which
considerably improves the approximation error.
We can then construct a matrix with orthonormal columns (ON) $\mtx{Q}$ of size
$m \times (k+p)$ via a \QR~factorization of $\mtx{Y}$. If $\mtx{Y}$ captures
a good portion of the range of $\mtx{A}$ (assuming $k$ is sufficiently large relative
to the numerical rank of $\mtx{A}$),
then we expect that $\mtx{Q} \mtx{Q}^{*} \mtx{A} \approx \mtx{A}$. The idea
is to compute the factorization of the smaller $(k+p) \times n$
product matrix $\mtx{Q}^* \mtx{A}$ and then multiply the result by $\mtx{Q}$ to
obtain an approximation of $\mtx{A}$.
Note that while $\mtx{Q}^* \mtx{Q} = \mtx{I}$, the matrix
product $\mtx{Q} \mtx{Q}^*$ multiplies out to
the identity only in the case that $\mtx{Q}$ is a square orthogonal matrix.
The product $\mtx{P} = \mtx{Q} \mtx{Q}^*$ is a projector onto the range of
$\mtx{Q}$ and hence onto that of $\mtx{Y}$ (where we assume that $\mtx{Q}$ is
obtained via a compact \QR~factorization of $\mtx{Y}$). To verify this, it is
easy to show that $\mtx{P}^2 = \mtx{P}$ and for any
$\vct{v} \in \range(\mtx{Q})$, $\mtx{P}\vct{v} = \vct{v}$. In the case
that $\mtx{Y}$ (and hence $\mtx{Q}$ obtained from $\mtx{Y}$) captures the entire range
of $\mtx{A}$, we have equality of
$\mtx{Q} \mtx{Q}^* \mtx{A}$ and $\mtx{A}$, per the lemma below:
\begin{lemma}
\label{lemma:rangeQ}
Let $\mtx{A} \in \mathbb{R}^{m \times n}$ and let $\mtx{Q} \in \mathbb{R}^{m \times r}$ be an orthonormal
matrix. Then the following are equivalent:
\begin{itemize}
\item[(1)]$\range(\mtx{A}) \subseteq \range(\mtx{Q})$
\item[(2)]$\mtx{A} = \mtx{Q} \mtx{Q}^* \mtx{A}$
\end{itemize}
\end{lemma}
\begin{proof}
Assume $(1)$ holds. Then this implies that there exists a matrix
$\mtx{S} \in \mathbb{R}^{r \times n}$
such that $\mtx{A} = \mtx{Q} \mtx{S}$. It follows that:
\begin{equation*}
\mtx{Q} \mtx{Q}^* \mtx{A} = \mtx{Q} \mtx{Q}^* \mtx{Q} \mtx{S} = \mtx{Q} \mtx{S} = \mtx{A},
\end{equation*}
since $\mtx{Q}^* \mtx{Q} = I$. Hence, $(1) \implies (2)$.
Next, assume $(2)$ holds. Then:
\begin{equation*}
\mtx{A} = \mtx{Q} \mtx{Q}^* \mtx{A} = \mtx{Q} \left( \mtx{Q}^* \mtx{A} \right) \implies \range(\mtx{A}) \subseteq \range(\mtx{Q}).
\end{equation*}
Hence, $(2) \implies (1)$.
\end{proof}

An extension of lemma \ref{lemma:rangeQ} states that when
$\mathcal{R}(\mtx{Q})$ is close to $\mathcal{R}(\mtx{A})$ (that is, when $\mtx{Q}$ captures
much of the range of $\mtx{A}$), then $\mtx{Q} \mtx{Q}^{*} \mtx{A} \approx \mtx{A}$. To make
this statement more precise, we summarize
here some results from \cite{2011_martinsson_randomsurvey}. Suppose we take a
GIID (Gaussian independent identically distributed) matrix $\mtx{\Omega}$ of size
$n \times l$, where $l = k + p$ with $k$ being the rank of
the approximation we seek and $p$ being a small oversampling parameter. We may
split the \SVD~of $\mtx{A}$ as:
\begin{equation*}
\mtx{A}
\quad = \quad \kbordermatrix{&
k\\
m &\mtx{U} } \ \kbordermatrix{&
k & n-k\\
&\mtx{\Sigma}_1& \mtx{0} \\ &\mtx{0}& \mtx{\Sigma}_2}
\kbordermatrix{
 &n
\\k&  \mtx{V}_1^{*}
\\n-k&  \mtx{V}_2^{*}
} \
\end{equation*}
Now let $\mtx{\Omega}_1 = \mtx{V}_1^* \mtx{\Omega}$ and $\mtx{\Omega}_2 = \mtx{V}_2^* \mtx{\Omega}$.
We may use these to write the sample matrix $\mtx{Y}$ and the corresponding ON matrix as:
\begin{equation*}
\mtx{Y}
\quad = \mtx{A} \mtx{\Omega} = \quad
\kbordermatrix{&
k\\
m &\mtx{U} } \
\kbordermatrix{&
l \\
k &\mtx{\Sigma}_1 \mtx{\Omega}_1 \\ n-k & \mtx{\Sigma}_2 \mtx{\Omega}_2} \quad, \quad \mtx{Q} = \orth(\mtx{A}).
\end{equation*}
Since $\range(\mtx{Q}) = \range(\mtx{Y})$ (\orth~is implemented via a compact \QR~factorization), the projector onto the
range of $\mtx{Y}$, $\mtx{P}_Y$, is equivalent to the projector onto the range of $\mtx{Q}$.
Hence, $\mtx{P}_Y = \mtx{Q} \mtx{Q}^*$ and
$\mtx{A} - \mtx{Q} \mtx{Q}^* \mtx{A} = \mtx{A} - \mtx{P}_Y \mtx{A} = (\mtx{I} - \mtx{P}_Y) \mtx{A}$.
In \cite{2011_martinsson_randomsurvey}, it is shown that if $\mtx{\Omega}_1$ has full row
rank then the approximation error satisfies:
\begin{equation}
\label{eq:py_error_bnd1}
\| (\mtx{I} - \mtx{P}_Y) \mtx{A} \| \leq \| \mtx{\Sigma}_2 \|^2 + \|\mtx{\Sigma}_2 \mtx{\Omega}_2 \mtx{\Omega}_1^{\dagger}\|^2,
\end{equation}
in both the spectral and Frobenius norms. In \eqref{eq:py_error_bnd1},
when $\mtx{A}$ has precisely rank $k$, then $\mtx{\Sigma}_2 = 0$ and the right hand side is zero so that
$\mtx{Q} \mtx{Q}^* \mtx{A} = \mtx{A}$ and so $\mtx{Y}$ will capture the range of $\mtx{A}$.
As $k$ approaches the rank of $\mtx{A}$, $\mtx{Y}$ will capture more and more of $\range(\mtx{A})$ and
the matrix product $\mtx{Q} \mtx{Q}^* \mtx{A}$
will approach $\mtx{A}$.
If we have an ON matrix $\mtx{Q}$ such that:
\begin{equation}
\label{eq:QQ_bound}
\|(\mtx{I} - \mtx{Q} \mtx{Q}^{*})\mtx{A}\| < \epsilon,
\end{equation}
then we can perform a factorization
of $\mtx{Q}^{*} \mtx{A}$ which is of size $(k+p) \times n$ and
multiply by $\mtx{Q}$ to obtain an approximate factorization of $\mtx{A}$ with the
same approximation error (in the same norm) as indicated by \eqref{eq:QQ_bound}.
Assuming $k \ll \min(m,n)$,
the matrix $\mtx{Q}^{*} \mtx{A}$ is substantially smaller than $\mtx{A}$.
For example, we can perform the SVD (of full rank $k+p$) of $\mtx{B} = \mtx{Q}^{*} \mtx{A}$
and then multiply by $\mtx{Q}$ to get an approximate SVD of $\mtx{A}$:
\begin{equation*}
\mtx{B} = \tilde{\mtx{U}} \mtx{D} \mtx{V}^* \implies \mtx{A} \approx (\mtx{Q} \tilde{\mtx{U}}) \mtx{D} \mtx{V}^*
\end{equation*}
Notice that instead of performing the SVD of $\mtx{B}$, we can also perform e.g. a
pivoted QR factorization of $\mtx{B}$ to get an approximate pivoted QR
factorization of $\mtx{A}$:
\begin{equation*}
\mtx{B} \mtx{P} = \tilde{\mtx{Q}} \mtx{R} \implies \mtx{A} \mtx{P} \approx (\mtx{Q} \tilde{\mtx{Q}}) \mtx{R}
\end{equation*}
For either approximation, the bound $\|(\mtx{I} - \mtx{Q} \mtx{Q}^{*})\mtx{A}\|$
is the error bound between the obtained factorization and
the original matrix, since the factorization of $\mtx{B} = \mtx{Q}^* \mtx{A}$ is exact.
In the software we provide, we implement two types of randomized routines: plain randomized
and block randomized. The plain randomized routines form $\mtx{Q}$ out of the matrix
of samples $\mtx{Y}$ using a single \QR~factorization of the matrix
$\mtx{Y} = \mtx{A} \mtx{\Omega}$ (or $\left(\mtx{A} \mtx{A}^*\right)^q \mtx{A}$, as
we discuss in \ref{sec:power_scheme}). This can involve operations with large matrices
depending on the desired rank $k$. The block randomized routines form the
matrices $\mtx{Q}$ and $\mtx{B}$ in a blocked fashion using multiple \QR~factorizations
and matrix multiplications of smaller matrices controlled by a block size parameter,
by employing the iterative procedure from \cite{2015arXiv150307157M}.

\subsection{Column norm preservation}
Another useful aspect of randomized sampling concerns the preservation of column norm
variations in a matrix. Suppose $\mtx{A}$ is $m \times n$ and we draw an
$l \times m$ GIID matrix $\tilde{\mtx{\Omega}}$. Suppose we then form
the $l \times n$ matrix $\mtx{Z} = \tilde{\mtx{\Omega}} \mtx{A}$.
We can then derive the following
expectation result relating the column norms of $\mtx{A}$ and $\mtx{Z}$:
\begin{equation}
\label{eq:col_norm_relations}
E\left[\frac{\|\mtx{Z}(:,j)\|^2}{\|\mtx{A}(:,j)\|^2}\right] = l.
\end{equation}
The result follows from the following lemma, where we make use of the 
construction mentioned in \cite{duersch2015true}.
\begin{lemma}
\label{lem:Omega_a}
Let $\tilde{\mtx{\Omega}} \in \mathbb{R}^{l \times m}$ be a matrix with GIID entries. Then
for any $\vct{a} \in \mathbb{R}^m$ we have that
$\mathrm{E}[\frac{\|\tilde{\mtx{\Omega}} \vct{a}\|^2}{\|\vct{a}\|^2}] = l$ and 
$\mathrm{Var}[\frac{\|\tilde{\mtx{\Omega}} \vct{a}\|^2}{\|\vct{a}\|^2}] = 2l$.
\end{lemma}
\begin{proof}
Since $\tilde{\mtx{\Omega}}$ has GIID entries, $\tilde{\mtx{\Omega}}_{ij} \sim \mathbb{N}(0,1)$
and hence, $\|\tilde{\mtx{\Omega}}(:,j)\|^2 = \sum_{i=1}^l  \tilde{\mtx{\Omega}}^2_{ij}$,
being the sum of squares of $m$ iid Gaussian random variables is distributed
as $\chi^2_m$ (Chi-squared distribution of degree $m$). It follows that
$E[\|\tilde{\mtx{\Omega}} \vct{e}_j\|^2] = m$. Note that if
$\tilde{\mtx{Q}} = \mathbb{R}^{m \times m}$
is an orthogonal matrix (that, is $\tilde{\mtx{Q}} \tilde{\mtx{Q}}^{*} = \tilde{\mtx{Q}}^{*} \tilde{\mtx{Q}} = \mtx{I}$),
then the matrix $\mtx{W} = \tilde{\mtx{\Omega}} \tilde{\mtx{Q}}$ which is $l \times n$,
is also GIID.
This means that $\mtx{W}_{ij} \sim \mathbb{N}(0,1)$. It follows that
$\|\mtx{W}_{(:,j)}\|^2 = \displaystyle\sum_{i=1}^l \mtx{W}_{ij}^2 \sim \chi^2_l$. Hence,
$E\left[ \|\mtx{W}_{(:,j)}\|^2 \right] = l$. Now let $\vct{u}_a$ be the unit vector
corresponding to $\vct{a}$, that is $\vct{u}_a = \frac{\vct{a}}{\|\vct{a}\|}$ and let
$\mtx{Q}_{\perp} \in \mathbb{R}^{m \times (m-1)}$ be a complement giving a full orthogonal
basis so that $\tilde{\mtx{Q}} = \begin{bmatrix}\vct{u}_a & \mtx{Q}_{\perp}\end{bmatrix}$ is an
orthogonal matrix. Since
$\tilde{\mtx{Q}}^{*} \tilde{\mtx{Q}} = I$, we must have:
\begin{equation*}
\tilde{\mtx{Q}}^{*} \tilde{\mtx{Q}} = \begin{bmatrix} \vct{u}_a^{*} \\ \mtx{Q}_{\perp}^{*} \end{bmatrix} \begin{bmatrix} \vct{u}_a & \mtx{Q}_{\perp} \end{bmatrix} = \begin{bmatrix} \vct{u}_a^{*} \vct{u}_a & \vct{u}_a^{*} \mtx{Q}_{\perp} \\
\mtx{Q}_{\perp}^{*} \vct{u}_a & \mtx{Q}_{\perp}^{*} \mtx{Q}_{\perp} \end{bmatrix} =
\begin{bmatrix} \vct{1} & \vct{0} \\ \vct{0} & \mtx{I}_{m-1} \end{bmatrix}
\end{equation*}
It follows that:
\begin{equation*}
\tilde{\mtx{\mtx{Q}}}^{*} \vct{a} = \begin{bmatrix} \vct{u}_a^{*} \vct{a} \\ \mtx{\mtx{Q}}_{\perp}^{*} \vct{a} \end{bmatrix} =
\begin{bmatrix} \vct{u}_a^{*} \vct{a} \\ \vct{0} \end{bmatrix} = \begin{bmatrix} \|\vct{a}\|_2 \\ \vct{0} \end{bmatrix}
= \|\vct{a}\|_2 \vct{e}_1 \quad \mbox{where} \quad \vct{e}_1 = \begin{bmatrix} \vct{1} \\ \vct{0} \end{bmatrix}
\end{equation*}
Finally, we look at the quotient $\frac{\|\tilde{\mtx{\Omega}} \vct{a}\|_2}{\|\vct{a}\|_2}$. The
numerator $\tilde{\mtx{\Omega}} \vct{a} = \tilde{\mtx{\Omega}} \tilde{\mtx{\mtx{Q}}} \tilde{\mtx{\mtx{Q}}}^{*}\vct{a} = \mtx{W} \tilde{\mtx{\mtx{Q}}}^{*} \vct{a}$
where $\mtx{W} = \tilde{\mtx{\Omega}} \tilde{\mtx{\mtx{Q}}}$ is GIID since
$\tilde{\mtx{\mtx{Q}}}$ is orthogonal
and $\tilde{\mtx{\mtx{Q}}}^{*} \vct{a} = \|\vct{a}\|_2 \vct{e}_1$. Thus,
$\tilde{\mtx{\Omega}} \vct{a} = \mtx{W} \tilde{\mtx{\mtx{Q}}}^{*} \vct{a} =
\mtx{W} \|\vct{a}\|_2 \vct{e}_1 = \|\vct{a}\|_2 \mtx{W}_{(:,1)}$. It follows that:
\begin{equation*}
\frac{\|\tilde{\mtx{\Omega}} \vct{a}\|^2}{\|\vct{a}\|^2} = \|\mtx{W}_{(:,1)}\|^2 \sim \chi_l^2 \quad \implies \quad
\mathrm{E}\left[\frac{\|\tilde{\mtx{\Omega}} \vct{a}\|^2}{\|\vct{a}\|^2}\right] = l, \quad \mathrm{Var}\left[\frac{\|\tilde{\mtx{\Omega}} \vct{a}\|^2}{\|\vct{a}\|^2}\right] = 2l.
\end{equation*}
\end{proof}
From Lemma \ref{lem:Omega_a}, with $\vct{a} = \mtx{A} \vct{e}_j = \mtx{A}(:,j)$,
relation \eqref{eq:col_norm_relations} follows. By virtue of this relation, we have that 
\begin{equation*}
\mathrm{E}\left[\|\mtx{Z}(:,i) - \mtx{Z}(:,j)\|^{2}\right] = l\,\|\mtx{A}(:,i) - \mtx{A}(:,j)\|^{2},
\end{equation*}
so that the smaller matrix $\mtx{Z} = \tilde{\mtx{\Omega}} \mtx{A}$ derived from $\mtx{A}$ 
preserves the column norm variations
of the original matrix $\mtx{A}$, up to a constant multiple dependent on $l$.
At least nominally, this explains why the pivoting matrix in a pivoted \QR~decomposition of 
$\mtx{Z}$ is expected to also work for the larger $\mtx{A}$, which provides some motivation for
the randomized \ID~algorithm which we later present.

\subsection{Power sampling scheme}
\label{sec:power_scheme}
In our previous discussion in \ref{sec:randomized_algorithms},
the matrix $\mtx{Q}$ which we use to project the
matrix $\mtx{A}$ into a lower dimensional space is constructed via compact \QR~factorization
of $\mtx{Y}$:
\begin{equation*}
\mtx{Y} = \mtx{A} \mtx{\Omega} \rightarrow \mtx{Q} = \orth(\mtx{Y}).
\end{equation*}

We now describe a so called \textit{power sampling scheme}, which improves the error bound
on the low rank approximation when the tail singular values
(i.e. $\sigma_{k+1}, \dots, \sigma_{r}$) are significant.
It is also helpful in cases where a matrix $\mtx{A}$ with singular values
$\sigma_1,\dots,\sigma_r$ has one or more large sequence
of singular values $\sigma_{i_1},\dots,\sigma_{i_p}$ with
$1 \leq i_1 \leq i_p \leq r$,
where the singular values decrease slowly in magnitude. In such cases, it substantially
helps to use a power scheme when sampling the range of the matrix $\mtx{A}$.
Instead of forming $\mtx{A} \mtx{\Omega}$, we form the matrix
$\mtx{Y} = \left( (\mtx{A} \mtx{A}^{*})^q \mtx{A} \right) \mtx{\Omega}$, where $q\geq1$ is an integer parameter.
We then set $\mtx{Q} = \orth(\mtx{Y})$.
Note that $\mtx{A}$ and $(\mtx{A} \mtx{A}^{*})^q \mtx{A}$ have the same eigenvectors
and related eigenvalues. Plugging in the SVD $\mtx{A} = \mtx{U} \mtx{\Sigma} \mtx{V}^{*}$, we have:
\begin{eqnarray*}
&& \mtx{A} \mtx{A}^{*} = \mtx{U} \mtx{\Sigma}^2 \mtx{U}^{*} \implies (\mtx{A} \mtx{A}^{*})^2 = \mtx{U} \mtx{\Sigma}^2 \mtx{U}^{*} \mtx{U} \mtx{\Sigma}^2 \mtx{U}^{*} = \mtx{U} \mtx{\Sigma}^4 \mtx{U}^{*} \implies (\mtx{A} \mtx{A}^{*})^q = \mtx{U} \mtx{\Sigma}^{2q} \mtx{U}^{*}  \\
&& \implies (\mtx{A} \mtx{A}^{*})^q \mtx{A} = \mtx{U} \mtx{\Sigma}^{2q} \mtx{U}^{*} \mtx{U} \mtx{\Sigma} \mtx{V}^{*} = \mtx{U} \mtx{\Sigma}^{2q + 1} \mtx{V}^{*}
\end{eqnarray*}
When $\mtx{A}$ is such that its trailing singular values decay slowly, the
matrix $(\mtx{A} \mtx{A}^{*})^q \mtx{A}$ with $q \geq 1$ has significantly faster
rate of singular value decay.

The matrix $\mtx{Z} = (\mtx{A} \mtx{A}^{*})^q \mtx{A} \mtx{\Omega}$ can be built up by means of the
following iterative procedure:
\begin{center}
\begin{minipage}{100mm}
\begin{tabbing}
\= \hspace{10mm} \= \hspace{5mm} \= \hspace{5mm} \= \hspace{5mm} \\ \kill
\> (1) \> $\mtx{Y} = \mtx{A} \mtx{\Omega}$\\
\> (2) \> \textbf{for} $i = 1:q$\\
\> (3) \> \> $\mtx{Y} \leftarrow \mtx{A}^{*} \mtx{Y}$\\
\> (4) \> \> $\mtx{Y} \leftarrow \mtx{A} \mtx{Y}$\\
\> (5) \> \textbf{end}
\end{tabbing}
\end{minipage}
\end{center}
In practice, we may wish to orthonormalize before multiplications with $\mtx{A}$ to prevent
repeatedly multiplying matrices having singular values greater than one. In cases
where very high computational precision is required (higher than
$\epsilon_{mach}^{\frac{1}{2q + 1}}$, where $\epsilon_{mach}$ is the machine
precision \cite{2015arXiv150307157M}), one
typically needs to orthonormalize the sampling matrix in
between each multiplication, resulting in the scheme:
\begin{center}
\begin{minipage}{100mm}
\begin{tabbing}
\= \hspace{10mm} \= \hspace{5mm} \= \hspace{5mm} \= \hspace{5mm} \\ \kill
\> (1) \> $\mtx{Y} = \mtx{A} \mtx{\Omega}$\\
\> (2) \> \textbf{for} $i = 1:q$\\
\> (3) \> \> $\mtx{Y} \leftarrow \mtx{A}^{*} \texttt{orth}(\mtx{Y})$\\
\> (4) \> \> $\mtx{Y} \leftarrow \mtx{A} \texttt{orth}(\mtx{Y})$\\
\> (5) \> \textbf{end}
\end{tabbing}
\end{minipage}
\end{center}
In many cases, $\orth$ does not need to be performed twice at each iteration. In the software, we use a parameter $s$
which controls how often the orthogonalization is done ($s = 1$ corresponds to the
above case, $s=2$ corresponds to doing one $\texttt{orth}$ operation per iteration, and
greater values for $s$ correspond to correspondingly less frequent orthogonalization).
In the randomized algorithm we present for the \ID~computation, we multiply by a GIID
matrix from the left. A similar power scheme in that case (which we later discuss)
also offers similar benefits.

\subsection{Adaptive rank approximation algorithms}
\label{sec:adaptive_rank_approx}
A major challenge in constructing suitable low rank approximations via randomized schemes is
in the construction of an ON matrix $\mtx{Q}$ such that
$\mtx{Q} \mtx{Q}^* \mtx{A} \approx \mtx{A}$, since the corresponding
decomposition can then be formed of the smaller matrix $\mtx{Q}^* \mtx{A}$. The simplest
approach to follow is to do as discussed before, forming $\mtx{Y} = \mtx{A} \mtx{\Omega}$
(we do not mention here the use of the performance improving power sampling scheme,
to which we will return momentarily).
The problem is that without advanced knowledge of the singular value
distribution of $\mtx{A}$, it is hard to guess an optimal rank $k$ parameter in the
size of the $n \times (k+p)$ matrix $\mtx{\Omega}$. If $k$ is chosen too small
relative to the numerical rank of $\mtx{A}$, then given
$\mtx{Q} = \orth(\mtx{Y})$, obtained via a compact \QR~factorization of $\mtx{Y}$,
$\mtx{Q} \mtx{Q}^* \mtx{A}$ will not be close to $\mtx{A}$. If instead $k$ is chosen
too large, then the matrix $\mtx{B} = \mtx{Q}^* \mtx{A}$, once computed
(itself a possibly overly expensive operation because of the \QR~on needlessly large $\mtx{Y}$)
will be large and not much time and memory savings could be obtained by performing the
wanted factorization of $\mtx{B}$ instead of the original $\mtx{A}$.

One way to proceed, is to start with a small $\mtx{Y}$, and then increase the size of
$\mtx{Y}$ by a block of samples at a time (by appending to $\mtx{Y}$ a matrix
$\mtx{Y}_{\textrm{sm}} = \mtx{A} \mtx{\Omega}_{\textrm{sm}}$ where
$\mtx{\Omega}_{\textrm{sm}}$ is a small $n \times \textrm{kstep}$ random Gaussian matrix).
The new enlarged matrix $\mtx{Y} = [\mtx{Y}, \mtx{Y}_{\textrm{sm}}]$
can then be orthogonalized to obtain a larger $\mtx{Q}$ using e.g. $\mtx{Q} = qr(\mtx{Y},0)$.
One can stop when the generated $\mtx{Q}$ becomes large enough so that
$\|\mtx{Q}\mtx{Q}^* \mtx{A} - \mtx{A}\|$ becomes sufficiently small.
The main problem with this approach is its inefficiency, given that the \QR~factorization
needs to be applied repeatedly to an increasingly larger $\mtx{Y}$.

We now discuss two more efficient algorithms for the automatic construction of suitable matrices
$\mtx{Q}$ and $\mtx{B} = \mtx{Q}^{*} \mtx{A}$ given a parameter $\epsilon>0$.
The algorithms are presented side by side in Figure
\ref{fig:Algorithms1and2}, taken from \cite{2015arXiv150307157M}.
The first algorithm $\QB_1$ is the single vector version, where $\mtx{Q}$ is built up
a column at a time, and the second algorithm $\QB_2$ is a blocked method, where
$\mtx{Q}$ is built up more rapidly, using blocks of vectors at each step. In the second
algorithm, the power scheme is also employed.
For both methods, on exit, we have that \eqref{eq:QQ_bound} holds.
Let us first analyze the simpler single vector method.
\begin{lemma}
At the end of iteration $j$ of Algorithm $\QB_1$, we have:
\begin{equation}
\label{eq:adaptiveQ_algs_quantities}
\mtx{A}^{(j)} = (\mtx{I} - \mtx{Q}_j \mtx{Q}_j^*) \mtx{A} \quad \mbox{and} \quad
\mtx{B}_j = \mtx{Q}_j^* \mtx{A}
\end{equation}
\end{lemma}
\begin{proof}
The results can be established by induction. Notice first that
$\|\vct{q}_j\| = 1 = \vct{q}_j^* \vct{q}_j$ for all $j$. Initially,
$\vct{q}_1 \in \range(\mtx{A})$, $\mtx{Q}_1 = [\vct{q}_1]$,
$\mtx{B}_1 = [\vct{b}_1] = [\vct{q}_1^* \mtx{A}] = \mtx{Q}_1^* \mtx{A}$.
After the first iteration,
$\mtx{A}^{(1)} = \mtx{A}^{(0)} - \vct{q}_1 \vct{b}_1 = \mtx{A} - \vct{q}_1 \vct{q}_1^* \mtx{A} = (\mtx{I} - \mtx{Q}_1 \mtx{Q}_1^*) \mtx{A}$.
Next, $\vct{q}_2 \in \range(\mtx{A}^{(1)}) = \range\left((\mtx{I} - \vct{q}_1 \vct{q}_1^*)\mtx{A}\right) \in \range(\mtx{I} - \vct{q}_1 \vct{q}_1^*)$ which
implies $\vct{q}_2 = (\mtx{I} - \vct{q}_1 \vct{q}_1^*)\vct{\mu}$ for some vector $\vct{\mu}$.
Thus, $\vct{q}_2^* \vct{q}_1 = \vct{\mu}^* (\mtx{I} - \vct{q}_1 \vct{q}_1^*) \vct{q}_1 = \vct{0}$. It follows that:
\begin{equation*}
\mtx{A}^{(2)} = \mtx{A}^{(1)} - \vct{q}_2 \vct{b}_2 = (\mtx{I} - \vct{q}_1 \vct{q}_1^*)\mtx{A} - \vct{q}_2 \vct{q}_2^* (\mtx{I} - \vct{q}_1 \vct{q}_1^*)\mtx{A} =
\mtx{A} - \vct{q}_1 \vct{q}_1^* \mtx{A} - \vct{q}_2 \vct{q}_2^* \mtx{A} =
(\mtx{I} - \mtx{Q}_2 \mtx{Q}_2^*)\mtx{A}.
\end{equation*}
In general, we have that $(\mtx{I} - \vct{q}_j \vct{q}_j^*) \perp \range(\vct{q}_j)$ and
that $\vct{q}_i \perp \vct{q}_j$ for $i \neq j$.
Let us assume that $\mtx{A}^{(j)} = (\mtx{I} - \mtx{Q}_j \mtx{Q}_j^*) \mtx{A}$.
It follows that:
\begin{eqnarray*}
\mtx{A}^{(j+1)} &=& \mtx{A}^{(j)} - \vct{q}_{(j+1)} \vct{q}^{*}_{(j+1)} \mtx{A}^{(j)} =
(\mtx{I} - \vct{q}_{(j+1)} \vct{q}^{*}_{(j+1)}) \mtx{A}^{(j)} =
(\mtx{I} - \vct{q}_{(j+1)} \vct{q}^{*}_{(j+1)}) (\mtx{I} - \mtx{Q}_j \mtx{Q}_j^*) \mtx{A} \\
&=& (\mtx{I} - \mtx{Q}_j \mtx{Q}_j^* - \vct{q}_{(j+1)} \vct{q}^{*}_{(j+1)}) \mtx{A} =
(\mtx{I} - \mtx{Q}_{j+1} \mtx{Q}_{j+1}^*) \mtx{A}
\end{eqnarray*}
Similarly, if we assume $\mtx{B}_j = \mtx{Q}_j^* \mtx{A}$, then we have:
\begin{equation*}
\mtx{B}_{(j+1)} = \begin{bmatrix} \mtx{Q}_j^* \mtx{A} \\ \vct{q}_{(j+1)}^* \mtx{A}^{(j)} \end{bmatrix}
= \begin{bmatrix} \mtx{Q}_j^* \mtx{A} \\ \vct{q}_{(j+1)}^* (\mtx{I} - \mtx{Q}_j \mtx{Q}_j^*) \mtx{A} = \vct{q}_{(j+1)}^* \mtx{A} \end{bmatrix} = \begin{bmatrix} \mtx{Q}_j^* \\ \vct{q}_{(j+1)}^* \end{bmatrix} \mtx{A} = \mtx{Q}_{(j+1)}^* \mtx{A}
\end{equation*}
\end{proof}
Notice that we quit Algorithm $\QB_1$ precisely when
$\|\mtx{A}^{(j)}\| < \epsilon$ is small, so since we have shown that
$\mtx{A}^{(j)} = (\mtx{I} - \mtx{Q}_j \mtx{Q}_j^*) \mtx{A}$, we have
on exit that \eqref{eq:QQ_bound} holds.

By a similar argument, it is proved in \cite{2015arXiv150307157M} that on output, the quantities
in \eqref{eq:adaptiveQ_algs_quantities} hold for the blocked scheme $\QB_2$.
The blocked algorithm is substantially accelerated over the non-blocked version, since
$\mtx{Q}$ is amended at each iteration by a block of column vectors.
\begin{figure}[ht!]
\centerline{
\includegraphics[scale=0.25]{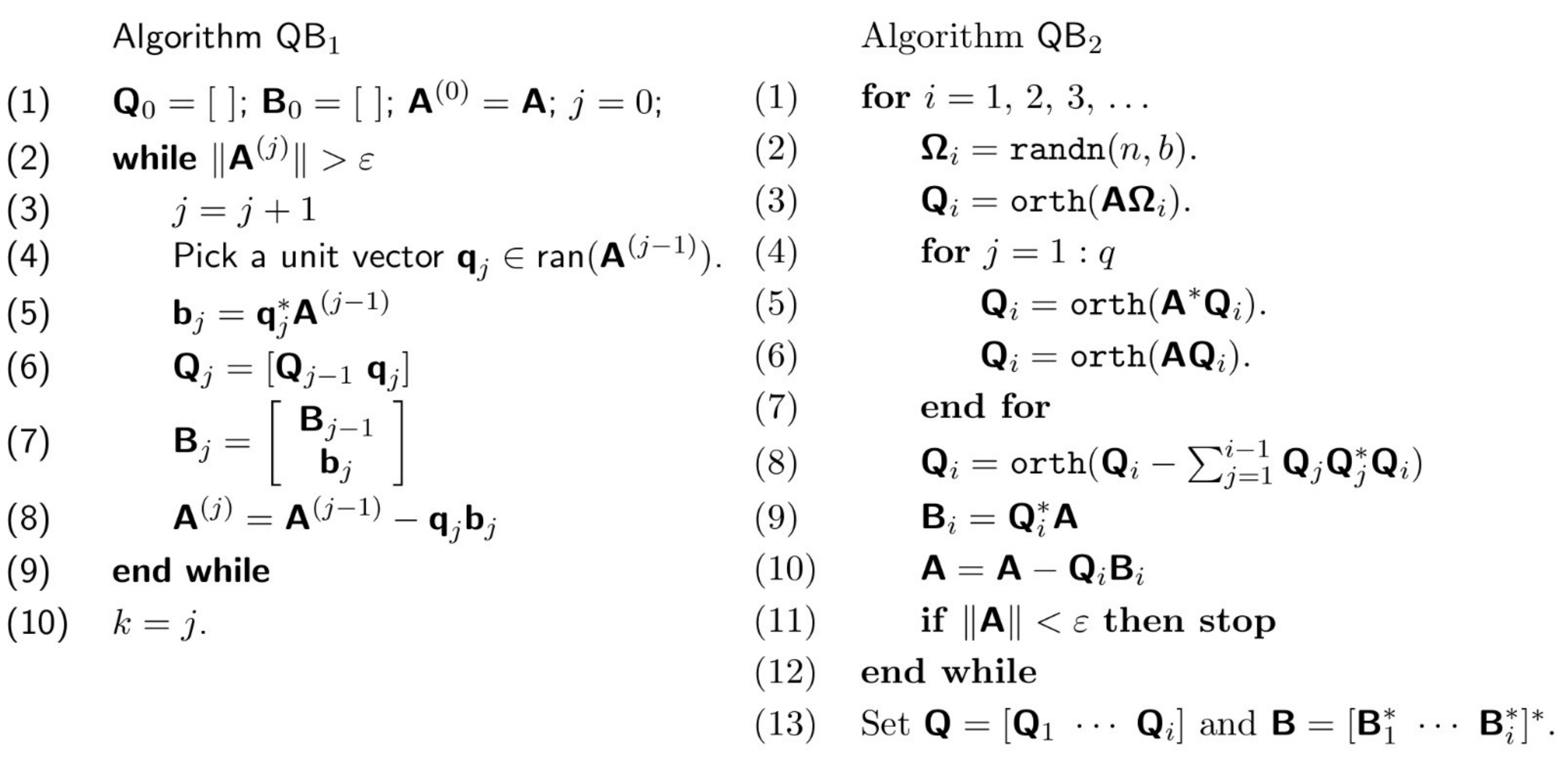}
}
\caption{Single vector and blocked algorithms for the adaptive construction of $\mtx{Q}$.\label{fig:Algorithms1and2}}
\end{figure}
The version we implement in RSVDPACK is based on this scheme,
giving a blocked algorithm for the construction
of matrices $\bar{\mtx{Q}}$ and $\bar{\mtx{B}}$, using a supplied tolerance $\epsilon > 0$. Here, we use the
\textit{bar} notation to represent a collection of block matrices,
where $\bar{\mtx{Q}}_i = \left[ \mtx{Q}_1, \dots, \mtx{Q}_i \right]$ and
$\bar{\mtx{B}}_i = \left[\mtx{B}^*_1 ; \dots ; \mtx{B}^*_i \right]^*$.
Pseudocode corresponding to scheme $\QB_2$ is shown in the appendix
as Algorithm \ref{algo:randpbQB2}.
Since the algorithm is blocked, the \orth~operation is performed only on matrices of $b$ columns,
with $b$ being the specified block size. The parameters $M$ and $\epsilon$ control how many
blocks are used: either the maximum specified by $M$ or when $\bar{\mtx{Q}}$ composed of a
certain number of blocks is large enough so
that $\|\bar{\mtx{Q}} \bar{\mtx{Q}}^{*} \mtx{A} - \mtx{A}\|$
becomes sufficiently small. Notice that the key to Algorithm \ref{algo:randpbQB2}
is to update the original matrix $\mtx{A}$. As we discuss below, this can be avoided
if desired with some accuracy tradeoffs. In Algorithm \ref{algo:randpbQB2},
if we assume $\mtx{Q}_i^{*} \mtx{Q}_j = 0$ for $i \neq j$ (which is proved in
\cite{2015arXiv150307157M}) and plug in $\mtx{B}_i = \mtx{Q}_i^{*} \mtx{A}^{(i)}$,
then at the $(i)$-th iteration, we obtain the recursive relations:
\begin{eqnarray*}
\mtx{A}^{(i)} &=& (\mtx{I} - \mtx{Q}_i \mtx{Q}_i^{*})\mtx{A}^{(i-1)} =
(\mtx{I} - \mtx{Q}_i \mtx{Q}_i^{*}) (\mtx{I} - \mtx{Q}_{i-1} \mtx{Q}_{i-1}^{*}) \mtx{A}^{(i-2)} =
(\mtx{I} - \mtx{Q}_i \mtx{Q}_i^{*} - \mtx{Q}_{i-1} \mtx{Q}_{i-1}^{*}) \mtx{A}^{(i-2)} \\
&=&\dots = (\mtx{I} - \bar{\mtx{Q}}_i \bar{\mtx{Q}}_i^*) \mtx{A}
\end{eqnarray*}
Thus, this implies that on line $(10)$ of Algorithm \ref{algo:randpbQB2}, we check
the condition $\|(\mtx{I} - \bar{\mtx{Q}}_i \bar{\mtx{Q}}_i^*) \mtx{A}\| < \epsilon$, so that if
we make $M$ large enough, we quit the loop precisely when
$\bar{\mtx{Q}}_i \bar{\mtx{Q}}_i^* \mtx{A} \approx \mtx{A}$ as in
\eqref{eq:QQ_bound}. The reorthonormalization procedure on line $(8)$ of Algorithm \ref{algo:randpbQB2} is
necessary to avoid loss of accuracy, but does not necessarily need to be performed
at each iteration.

In practice, the blocked scheme in Algorithm \ref{algo:randpbQB2} can often be
substantially simplified and still yield accurate results, particularly when $\mtx{A}$
has non-linear decay of its singular values. The main drawback of the scheme is
that the algorithm is essentially sequential, although each individual matrix-matrix
operation can of course take advantage of parallel processing. Another
disadvantage is that Algorithm \ref{algo:randpbQB2} requires the explicit updating
of the matrix $\mtx{A}$, which requires either to modify $\mtx{A}$ or make a copy of $\mtx{A}$
as another variable, which may not be desirable when $\mtx{A}$ is large. Algorithm
\ref{algo:randpbQB3} is a version which consists of several for loops, the first
two of which can be performed in parallel. This version also never requires
the \textit{orth} operation to be performed on a large matrix and never requires the
updating of $\mtx{A}$. In place of a big \QR, we perform a projection
operation on line $(13)$, which in practice, has the tendency to keep the columns
of $\mtx{Q}$ almost orthonormal with respect to each other. Algorithm \ref{algo:randpbQB2}
is more accurate than Algorithm \ref{algo:randpbQB3} and has the advantage that it can
satisfy \eqref{eq:QQ_bound} to a given $\varepsilon$ tolerance, but when the singular values
of $\mtx{A}$ decay rapidly, both methods give plausible results.

For very large matrices, we can make use of a blocked scheme.
We can proceed by subdividing $\mtx{A}$ into blocks along the rows. We assume here
that the number of blocks is a power of two. Without loss of generality, we assume
the use of four blocks:
\begin{equation*}
\mtx{A} = \begin{bmatrix}
\mtx{A}_1 \\
\mtx{A}_2 \\
\mtx{A}_3 \\
\mtx{A}_4
\end{bmatrix} \approx
\begin{bmatrix}
\mtx{Q}_1 \mtx{B}_1 \\
\mtx{Q}_2 \mtx{B}_2 \\
\mtx{Q}_3 \mtx{B}_3 \\
\mtx{Q}_4 \mtx{B}_4
\end{bmatrix} =
\begin{bmatrix}
\mtx{Q}_1 & 0 & 0 &0 \\
0 & \mtx{Q}_2 & 0 &0 \\
0 & 0 & \mtx{Q}_3 &0 \\
0 & 0 & 0 & \mtx{Q}_4 \\
\end{bmatrix}
\begin{bmatrix}
\mtx{B}_1 \\
\mtx{B}_2 \\
\mtx{B}_3 \\
\mtx{B}_4
\end{bmatrix}
\end{equation*}
We then perform \QB~factorizations on the blocks of the $\mtx{B}$ matrix:
\begin{equation*}
\mtx{M}^{(1)} = \begin{bmatrix} \mtx{B}_1 \\ \mtx{B}_2 \end{bmatrix} \approx \mtx{Q}_{12} \mtx{B}_{12} \quad \mbox{;} \quad \mtx{M}^{(2)} = \begin{bmatrix} \mtx{B}_3 \\ \mtx{B}_4 \end{bmatrix} \approx \mtx{Q}_{34} \mtx{B}_{34}
\end{equation*}
Finally, we perform a \QB~factorization on:
\begin{equation*}
\mtx{M}^{(3)} = \begin{bmatrix}
\mtx{B}_{12} \\
\mtx{B}_{34}
\end{bmatrix} \approx \mtx{Q}_{1234} \mtx{B}_{1234}
\end{equation*}
It follows that:
\begin{eqnarray*}
\mtx{A} &\approx& \begin{bmatrix}
\mtx{Q}_1 & 0 & 0 &0 \\
0 & \mtx{Q}_2 & 0 &0 \\
0 & 0 & \mtx{Q}_3 &0 \\
0 & 0 & 0 & \mtx{Q}_4 \\
\end{bmatrix}
\begin{bmatrix}
\mtx{Q}_{12} & 0 \\
0 & \mtx{Q}_{34}
\end{bmatrix}
\begin{bmatrix}
\mtx{B}_{12} \\
\mtx{B}_{34}
\end{bmatrix} \approx \begin{bmatrix}
\mtx{Q}_1 & 0 & 0 &0 \\
0 & \mtx{Q}_2 & 0 &0 \\
0 & 0 & \mtx{Q}_3 &0 \\
0 & 0 & 0 & \mtx{Q}_4 \\
\end{bmatrix}
\begin{bmatrix}
\mtx{Q}_{12} & 0 \\
0 & \mtx{Q}_{34}
\end{bmatrix} \mtx{Q}_{1234} \mtx{B}_{1234} \\
&=& \mtx{Q}^{(3)} \mtx{Q}^{(2)} \mtx{Q}^{(1)} \mtx{B}^{(1)} = \mtx{Q} \mtx{B}
\end{eqnarray*}
The benefit of this formulation is that the \QB~algorithm can be performed on smaller matrices
in parallel. In particular, we handle the decompositions of blocks $\mtx{A}_1,\dots,\mtx{A}_4$
in parallel, following which we can do in parallel the decompositions of matrices $\mtx{M}^{(1)}$
and $\mtx{M}^{(2)}$ and finally that of $\mtx{M}^{(3)}$. The reason for blocking $\mtx{A}$ along
the rows instead of columns is that we would like to keep the orthonormality of the
resulting matrix $\mtx{Q}$ which is done by working with block diagonal matrices.
Notice that the later factorizations of smaller matrices, can be replaced by full rank (library routine callable) \QR~factorizations in place of low rank \QB s. This hierarchical procedure is 
not implemented in RSVDPACK but is being explored as part of upcoming work.

\subsection{Randomized algorithms for the low rank \SVD}
\label{sec:randomized_algorithms_for_low_rank_svd}
For computing the low rank \SVD~of rank $k$ of matrix $\mtx{A} \in \mathbb{R}^{m \times n}$, the
basic randomized algorithm consists of the following steps:
\begin{itemize}
\item Form GIID $\mtx{\Omega} \in \mathbb{R}^{n \times l}$ with $l = k + p$.
\item Form sample matrix $\mtx{Y}$ of size $m \times l$ via $\mtx{Y} = \mtx{A} \mtx{\Omega}$.
\item Orthogonalize this set of samples forming the matrix $\mtx{Q} = \textrm{orth}(\mtx{Y})$.
\item Project the original matrix into a lower dimensional one: $\mtx{B} = \mtx{Q}^{*} \mtx{A}$,
where $\mtx{B}$ is $l \times n$, substantially smaller than $\mtx{A}$, which is $m \times n$.
\item Compute the SVD of the smaller matrix
$\mtx{B} = \tilde{\mtx{U}} \mtx{\mtx{\Sigma}} \mtx{V}^{*}$.
\item Form $\mtx{U} = \mtx{Q} \tilde{\mtx{U}}$.
\item Form the component matrices of the approximate rank-$k$ SVD of $\mtx{A}$ by setting:
\begin{equation*}
\mtx{U}_k = \mtx{U}(:,1:k), \mtx{\mtx{\Sigma}}_k = \mtx{\mtx{\Sigma}}(1:k,1:k), \mtx{V}_k = \mtx{V}(:,1:k),
\end{equation*}
so that the product $\mtx{U}_k \mtx{\mtx{\Sigma}}_k \mtx{V}^{*}_k \approx \mtx{A}$.
\end{itemize}
The first modification of the original algorithm
computes $\mtx{U}_k$ and $\mtx{V}_k$ without having to take the
\SVD~of the $l \times n$ matrix $\mtx{B}$, which may still be large if $n$ (the
number of columns of $\mtx{A}$) is large.
Instead of the \SVD, we can use the eigendecomposition of the smaller
$k \times k$ matrix $\mtx{B} \mtx{B}^{*}$. For this, we use the relations:
\begin{eqnarray*}
&& \mtx{B} = \tilde{\mtx{U}}_k \mtx{\Sigma}_k \mtx{V}^{*}_k = \displaystyle\sum_{i=1}^k \sigma_i \tilde{\vct{u}}_i \vct{v}_i^{*} \quad \mbox{;} \quad \mtx{B}^{*} = \mtx{V}_k \mtx{\Sigma}_k \tilde{\mtx{U}}_k^{*} \quad \mbox{;} \quad \mtx{B} \vct{v}_i = \sigma_i \tilde{\vct{u}}_i \\
&& \mtx{B} \mtx{B}^{*} = \left( \displaystyle\sum_{i=1}^k \sigma_i \tilde{\vct{u}}_i \vct{v}^{*}_i \right) \left( \displaystyle\sum_{j=1}^k \sigma_j \tilde{\vct{u}}_j \vct{v}^{*}_j \right)^{*} = \displaystyle\sum_{i,j=1}^k \sigma_i \sigma_j \tilde{\vct{u}}_i \vct{v}_i^{*} \vct{v}_j \tilde{\vct{u}}_j^{*} = \displaystyle\sum_{i=1}^k \sigma_i^2 \tilde{\vct{u}}_i \tilde{\vct{u}}_i^{*} = \tilde{\mtx{U}}_k \mtx{D}_k \tilde{\mtx{U}}_k^{*},
\end{eqnarray*}
where $\mtx{\Sigma}_k = \sqrt{\mtx{D}_k}$ element wise.
To compute the right eigenvectors $\vct{v}_i$, we can use the following relations, followed by truncation to the $k$ dominant components:
\begin{equation*}
\mtx{B}^{*} \tilde{\mtx{U}}_k = \mtx{V}_k \mtx{\Sigma}_k \tilde{\mtx{U}}_k^{*} \tilde{\mtx{U}}_k = \mtx{V}_k \mtx{\Sigma}_k \implies \mtx{V}_k = \mtx{B}^{*} \tilde{\mtx{U}}_k \mtx{\Sigma}_k^{-1},
\end{equation*}
assuming all the singular values in $\mtx{\Sigma}_k$ are above zero (which is the case for
$k$ smaller than the numerical rank of $\mtx{A}$). The eigendecomposition of the symmetric
matrix $\mtx{B} \mtx{B}^{*}$ can be easily performed with a library call or hand coded Lanczos
routine. Since this matrix is small ($l \times l$), the eigendecomposition
is not expensive. On the other hand, a simple yet possibly time and memory
consuming step when $\mtx{B}$ is large, is to carry out the matrix matrix multiplication
$\mtx{B} \mtx{B}^{*}$. When $\mtx{B}$ is large, this
step can be done in a column by column fashion via multiplication with
standard basis vectors $\mtx{B} (\mtx{B}^{*} e_j)$. A similar approach in this case is
useful for obtaining the columns of $\mtx{V}_k$. A disadvantage of using this method is that
the matrix $\mtx{B} \mtx{B}^{*}$ has essentially the square of the condition number of $\mtx{B}$.
As a result, very small singular values of $\mtx{A}$ near machine precision
may not be properly resolved. This is an issue only if $\mtx{A}$ is expected to have very
small singular values amongst $\sigma_1, \dots, \sigma_k$.

Another approach is to use a QR factorization of $\mtx{B}^{*}$ instead of
forming $\mtx{B} \mtx{B}^{*}$. Out of this we get a $l \times l$
matrix $\mtx{R}$ on which we perform the \SVD, instead of performing the \SVD~on the
$l \times n$ matrix $\mtx{B}$. This version is based on the following calculations.
Let us take the (economic form of the) \QR~factorization
$\mtx{B}^{*} = \hat{\mtx{Q}} \hat{\mtx{R}}$. Then $\hat{\mtx{R}}$ is $l \times l$ and taking the
\SVD~yields $\hat{\mtx{R}} = \hat{\mtx{U}} \hat{\mtx{\Sigma}} \hat{\mtx{V}}^{*}$.
\begin{equation*}
\mtx{A} \approx \mtx{Q} \mtx{Q}^{*} \mtx{A} = \mtx{Q} \mtx{B} = \mtx{Q} \hat{\mtx{R}}^{*} \hat{\mtx{Q}}^{*} = \mtx{Q} \hat{\mtx{V}} \hat{\mtx{\Sigma}} \hat{\mtx{U}}^{*} \hat{\mtx{Q}}^{*}
\end{equation*}
Thus, the low rank \SVD~components are
$\mtx{U}_k = \left(\mtx{Q} \hat{\mtx{V}}\right)(:,1:k), \mtx{\Sigma}_k = \mtx{\Sigma}(1:k,1:k),
\mtx{V}_k = \left(\hat{\mtx{Q}} \hat{\mtx{U}}\right)(:,1:k)$.
For both variations of the algorithms, the power scheme with sampling matrix
$(\mtx{A} \mtx{A}^{*})^q \mtx{A}$ may be used to improve performance for matrices
with significant tail singular values. Notice that in many software packages, the order
of the returned eigenvalues in an eigendecomposition is opposite to that of the singular
values in an \SVD. For this reason, following an eigendecomposition call,
instead of the first $k$ components, the last $k$
components must sometimes be extracted (for example, in Matlab) to correspond to the
extraction of the most significant (in absolute magnitude) components. The pseudocode for the
two methods is provided in the Appendix as algorithms \ref{algo:rank_rsvd1}
and \ref{algo:rank_rsvd2}.

For either method, the upper bound on the approximation error \cite{2011_martinsson_randomsurvey} with the
randomized low rank \SVD~algorithm can be large,
\begin{equation*}
\|\mtx{A} - \mtx{U}_k \mtx{\Sigma}_k \mtx{V}^{*}_k\|_2 \leq \sqrt{kn} \sigma_{k+1},
\end{equation*}
with respect to the optimal $\sigma_{k+1}$ (in the spectral norm) given in
Theorem \ref{thm:eckartyoung}. However, if we incorporate the use of the \textit{power samping scheme},
described in \ref{sec:power_scheme}, then the approximate upper bound  \cite{2011_martinsson_randomsurvey}  improves to:
\begin{equation*}
\|\mtx{A} - \mtx{U}_k \mtx{\Sigma}_k \mtx{V}^{*}_k\|_2 \leq (kn)^\frac{1}{2(2q + 1)} \sigma_{k+1}.
\end{equation*}
For sufficiently large $q$, this is substantially closer to the optimal value of $\sigma_{k+1}$
in Theorem \ref{thm:eckartyoung}. Numerical results with power sampling (using $q \geq 2$) are often very close to optimal, as
illustrated in Section \ref{sec:performance}.

As previously mentioned in \ref{sec:adaptive_rank_approx},
if a \QB~decomposition of $\mtx{A}$ to tolerance
$\epsilon$ is obtained, such that
$\|\mtx{Q}\mtx{B} - \mtx{A}\| < \epsilon$ (using Algorithm $\QB_2$, for instance), then
following one of the above procedures for computing the low rank \SVD, we would have the
approximation satisfy the same error bound:
$\|\mtx{U}_k \mtx{\Sigma}_k \mtx{V}^*_k - \mtx{A}\| < \epsilon$. This is a very useful
property of the \QB~decomposition and of the corresponding block randomized routines
in RSVDPACK.

\subsection{Randomized algorithms for the \ID}
\label{sec:randidalgs}

We now discuss the use of randomized sampling to speed up the \ID~computation.
Observe that in order to compute the column \ID~of a matrix, all we need is to
know the linear dependencies among the columns of $\mtx{A}$. When the singular values
of $\mtx{A}$ decay reasonably rapidly, we can determine these linear dependencies by
processing a matrix $\mtx{Y}$ of size $\ell \times n$, where $\ell$ can be much smaller
than $n$, similar to what we did in the randomized algorithm for computing
the low rank \SVD. In the randomized scheme for the \ID, we perform the
partial pivoted \QR~factorization of the smaller $\mtx{Y}$ rather than of $\mtx{A}$.
Mathematical justification for this is given by the relation between the column
norms of $\mtx{Y} = \mtx{\Omega} \mtx{A}$ and of $\mtx{A}$ in
\eqref{eq:col_norm_relations}. That is, the column norm variations of
$\mtx{A}$ are preserved by $\mtx{Y}$. (Notice that as discussed in 
section \ref{sec:randomized_algorithms}, by Lemma \ref{lem:Omega_a}, for
any $i,j$ between $1$ and $n$, if we take a difference of two columns of $\mtx{A}$
as $\vct{x} = \mtx{A}(:,i) - \mtx{A}(:,j)$
and $\vct{y} = \tilde{\mtx{\Omega}}\vct{x} = \mtx{Y}(:,i) - \mtx{Y}(:,j)$,
it follows that $\mathrm{E}\left[\|\vct{y}\|^2\right] = l \|\vct{x}\|^2$).
Suppose that we are given an $m\times n$ matrix $\mtx{A}$ and seek to compute
a column \ID, a two-sided \ID, or a \CUR~decomposition. We can perform this task as
long as we can identify an index vector $J_c$ as in \eqref{eq:Jc_index_vec}
and a basis matrix $\mtx{V} \in \mathbb{C}^{n\times k}$ such that
$$
\begin{array}{cccccccccc}
\mtx{A} &=& \mtx{A}(:,J_{\rm skel}) & \mtx{V}^{*} &+& \mtx{E}\\
m\times n && m\times k & k\times n && m\times n
\end{array}
$$
where $\mtx{E}$ is small. In Algorithm \ref{algo:rank_k_id} we found $J_c$ and $\mtx{V}$
by performing a column pivoted \QR~factorization of $\mtx{A}$. In order to do this via
randomized sampling, we proceed as above. First, we fix a small over-sampling parameter $p$,
then draw a $(k+p)\times m$ GIID matrix $\tilde{\mtx{\Omega}}$, and form the \textit{sampling matrix}
\begin{equation}
\label{eq:Y=OmegaA}
\begin{array}{cccccccc}
\mtx{Y} &=& \tilde{\mtx{\Omega}} & \mtx{A}.\\
(k+p)\times n && (k+p)\times m & m\times n
\end{array}
\end{equation}
We assume based on \eqref{eq:col_norm_relations}, that the space spanned by the rows of $\mtx{Y}$
contains the dominant $k$ right singular vectors of $\mtx{A}$ to high accuracy. This
is precisely the property we need in order to find both the vector $J$ and the
basis matrix $\mtx{V}$. All we need to do is to perform $k$ steps of a column pivoted
QR factorization of the sample matrix to form a partial QR factorization
$$
\begin{array}{cccccccc}
\mtx{Y}(:,{J}_c) &\approx& \mtx{Q} & \mtx{S}.\\
(k+p)\times n && (k+p)\times k & k\times n
\end{array}
$$
Then compute the matrix of expansion coefficients via $\mtx{T} = \mtx{S}(1:k,1:k)^{-1}\mtx{S}(1:k,(k+1):n)$,
or a stabilized version, as described in Section \ref{sec:QRandonesideID}.
The matrix $\mtx{V}$ is formed from $\mtx{T}$ as before,
resulting in Algorithm \ref{algo:rank_k_rand_id}.
The asymptotic cost of Algorithm \ref{algo:rank_k_rand_id} is $O(mnk)$,
like that of the randomized low rank \SVD.
However, substantial practical gain is achieved due to the fact that the
matrix-matrix multiplication is much faster than a column-pivoted \QR~factorization.
This effect gets particularly pronounced when a matrix is very large and is
stored either out-of-core, or on a distributed memory machine.

Finally, we notice that the power sampling scheme used for the low rank \SVD, is
also equally effective for the randomized \ID~scheme. In this case, since we multiply
by $\tilde{\mtx{\Omega}}$ on the left, we use the sampling matrix:
\begin{equation}
\label{eq:Ypower}
\mtx{Y} = \tilde{\mtx{\Omega}}\,\mtx{A}\,\bigl(\mtx{A}^{*}\mtx{A})^{q}.
\end{equation}
%
In cases where very high computational precision is required (higher than
$\epsilon_{\rm mach}^{1/(2q+1)}$, where $\epsilon_{\rm mach}$ is the machine
precision \cite{2014arXiv1412.8447V}), one typically needs to orthonormalize the sampling matrix in
between multiplications, resulting in:
\begin{center}
\begin{minipage}{100mm}
\begin{tabbing}
\= \hspace{10mm} \= \hspace{5mm} \= \hspace{5mm} \= \hspace{5mm} \\ \kill
\> (1) \> $\mtx{Y} = (\mtx{A}^{*}\mtx{\Omega}^{*})^{*} $\\
\> (2) \> \textbf{for} $i = 1:q$\\
\> (3) \> \> $\mtx{Y} \leftarrow  (\mtx{A}\texttt{orth}(\mtx{Y}^{*}))^{*}$\\
\> (4) \> \> $\mtx{Y} \leftarrow (\mtx{A}^{*}\texttt{orth}(\mtx{Y}^{*}))^{*} $\\
\> (5) \> \textbf{end}
\end{tabbing}
\end{minipage}
\end{center}
where as before, $\texttt{orth}$ refers to orthonormalization of the \textit{columns},
without pivoting.
The randomized algorithm for the
\ID~appears as Algorithm \ref{algo:rank_k_rand_id} in the appendix. The error for the
randomized \ID~approximation is lower bounded by the error in the truncated pivoted
\QR~factorization and typically stays reasonably close to this value when the power sampling
scheme with $q \geq 1$ is employed (see \cite{2014arXiv1412.8447V} and section
\ref{sec:performance}).

Analogously to the \SVD, we can compute the approximate rank $k$ \ID~given
matrix $\mtx{B} = \mtx{Q}^{*} \mtx{A}$, which allows us to use
the block algorithms discussed in \ref{sec:adaptive_rank_approx},
for the construction of
$\mtx{Q}$ and $\mtx{B}$ in computing the \ID. We can compute the
\ID~of $\mtx{B} = \mtx{Q}^* \mtx{A}$ to obtain:
\begin{equation*}
\mtx{B}(:,{J}_c) \approx \mtx{B}(:,{J}_c(1:k)) \mtx{S}^{*}
\end{equation*}
and then since, $\mtx{Q} \mtx{B} \approx \mtx{A}$, it follows from multiplication of
both sides by $\mtx{Q}$ that $\mtx{A}(:,{J}_c) \approx \mtx{A}(:,{J}_c(1:k)) \mtx{S}^{*}$.
Notice, however, that in contrast to the case of the low rank \SVD, if a \QB~decomposition to
tolerance $\epsilon$ is obtained, it does not imply that the resulting \ID~obtained
from $\mtx{B}$ will approximate $\mtx{A}$ to the same tolerance. The error is in practice larger.
From \cite{2011_martinsson_randomsurvey}, $(3.6)$, we have the bound:
\begin{equation*}
\|\mtx{A}(:,J_{\rm c}(1:k)) \mtx{V}^* - \mtx{A}\| \leq \left[1 + \sqrt{1 + 4k(n-k)}\right]\epsilon
\end{equation*}
for the \ID~obtained from $\mtx{B} = \mtx{Q}^* \mtx{A}$.

\subsection{Randomized algorithms for the \CUR~decomposition}

Once an approximate \ID~is obtained with a randomized scheme, the randomized algorithm for
the \CUR~proceeds as described in Section \ref{subsec:twosidedidandcur},
using the results of the two sided \ID~factorization. Notice that to form an approximate two sided \ID~only one application of the randomized
\ID~method is necessary. The subsequent \ID~is of a small matrix (see \eqref{eq:factorC}),
and does not need to employ randomization to retain efficiency.
The error is again lower bounded by the truncated pivoted \QR~factorization of the same rank. We illustrate some examples in Section \ref{sec:performance}.

\section{Developed Software}
\label{sec:software}
In this section, we describe the developed software which has been written to
implement the randomized algorithms for the computation of the low rank \SVD, \ID,
and \CUR~routines. We have developed codes for multi-core and GPU architectures.
In each case, we have used well known software libraries to
implement BLAS and certain LAPACK routines
and write wrappers for various BLAS and LAPACK operations (e.g. vector manipulation,
matrix multiplication, \QR, eigendecomposition, and \SVD~operations).
The codes are written using the C programming language and are built on top of the
Intel MKL, NVIDIA cuBLAS, and CULA libraries.
Since we created wrappers for most of the required matrix and vector functions, it is
not difficult to port the code to use other libraries for BLAS and LAPACK.
The codes use OpenMP, where possible, to speed up matrix-vector
operations on multi-core systems. Each code can load a matrix from disk stored using the
following simple binary format for dense matrices:
\lstset{language=C,
   keywords={break,case,catch,continue,else,elseif,end,for,function,
      global,if,otherwise,persistent,return,switch,try,while},
   basicstyle=\ttfamily,
   keywordstyle=\color{blue},
   commentstyle=\color{red},
   stringstyle=\color{dkgreen},
   numbers=left,
   numberstyle=\tiny\color{gray},
   stepnumber=1,
   numbersep=10pt,
   backgroundcolor=\color{white},
   tabsize=4,
   showspaces=false,
   showstringspaces=false}
\begin{lstlisting}
num_rows (int)
num_columns (int)
nnz (double)
...
nnz (double)
\end{lstlisting}
where the nonzeros are listed in the order of a double loop over the rows and columns
of the matrix. Note that even zero values are written in this format.
It is not difficult to extend the codes to support arbitrary matrix formats,
including those used for sparse matrices. We expect to add this functionality
in future releases. The matrix can be loaded using the supplied
function:
\begin{lstlisting}
matrix_load_from_binary_file(char *fname)
\end{lstlisting}

Below, we list the main available functions for \SVD, \ID, and \CUR~computations. These 
routines can be used from simple C driver programs. Example driver programs are provided for 
illustration with the source code. 
We also provide a mex file interface for some of these routines to use 
inside Matlab. 
\inputsourcecode{sourcecode.c}
We now describe the functions and their parameters. The functions for \SVD, \ID,
and \CUR~have a similar calling sequence. We provide a routine which computes the full
truncated decomposition, as well as routines for randomized and block randomized approximate
versions, either to a fixed rank or to a specified tolerance level. We now describe in detail
the routines for the low rank \SVD. The function:
\begin{lstlisting}
low_rank_svd_decomp_fixed_rank_or_prec(mat *A, int k, double TOL,
  int *frank, mat **U, mat **S, mat **V);
\end{lstlisting}
computes, using the full \SVD~of $\mtx{A}$, either the rank $k$ low rank \SVD~of $\mtx{A}$ by
truncating to the first $k$ components or computes the low rank \SVD~to given precision TOL
such that the singular value $\sigma_{k+1} < TOL$. The routine returns a fixed rank result
if the parameter $k > 0$ is supplied or computes the desired rank if $k\leq0$ and $TOL$ is
supplied instead.
The output rank is then written in $\textrm{frank}$. The output matrices $U,S,V$ contain the
components of the computed low rank \SVD. The routine:
\begin{lstlisting}
low_rank_svd_rand_decomp_fixed_rank(mat *A, int k, int p, int vnum,
  int q, int s, mat **U, mat **S, mat **V);
\end{lstlisting}
computes the low rank \SVD~of rank $k$ using randomized sampling. Here the parameter
$\textrm{vnum}$ (short for version number), indicates which routine to use for computation, either
the method using the $\mtx{B} \mtx{B}^*$ matrix, or the
\QR~decomposition method, as discussed in
Section \ref{sec:randomized_algorithms_for_low_rank_svd} ($\textrm{vnum}$ equal to 
$1$ for the \QR~method or to $2$ for  $\mtx{B} \mtx{B}^*$ method).
The parameter $p$ corresponds to oversampling (usually a small number $5 \leq p \leq 20$ relative to $k$).
Parameters $q$ and $s$ correspond to the power sampling scheme with $q$ being the power
and $s$ controlling the amount of orthogonalizations ($s=1$ corresponds to the most
orthogonalizations after each multiplication with $\mtx{A}$). Typically, one
uses $1 \leq q \leq 5$. Next, the routine:
\begin{lstlisting}
low_rank_svd_blockrand_decomp_fixed_rank_or_prec(mat *A, int k, int p,
   double TOL, int vnum, int kstep, int q, int s,
   int *frank, mat **U, mat **S, mat **V);
\end{lstlisting}
uses the block randomized algorithm to construct the approximate \QB~decomposition of
$\mtx{A}$, and from that construct the low rank \SVD, as explained in Section
\ref{sec:randomized_algorithms_for_low_rank_svd}.
The decomposition is computed either to specific rank $\approx (k + p)$ using a computed
number of blocks ($\approx \frac{(k+p)}{\textrm{kstep}}$) each of
size \textrm{kstep}, or is computed
such that $\|\mtx{A} - \mtx{Q}\mtx{B}\| < \textrm{TOL}$ if $k=0$ is supplied.
Parameter \textrm{vnum} again corresponds to the routine to use on $B$ as above and
$q,s$ correspond to the power sampling scheme. If rank $k>0$ is supplied, the resulting
decomposition components $\mtx{U},\mtx{S},\mtx{V}$ are truncated to rank $k$. Otherwise,
in \textrm{TOL} model, $k$ and the output parameter \textrm{frank} is set to
the output size of $\mtx{B}$ (which depends on the supplied tolerance \textrm{TOL}).

For the \ID~and \CUR~routines, TOL mode is based on \eqref{eqn:qrtrunc}. When $k > 0$ is not
satisfied, the tolerance based algorithms determine the \QR~decomposition rank such
that $\|\mtx{S}_{22}\| < TOL$.
In the block rand scheme, the fraction $\frac{k}{k+p}$ is used to determine the truncation
factor of the pivoted \QR~decomposition of \mtx{B}, whose size depends on the supplied
tolerance \textrm{TOL}. The output size of the \QR~components is written to \textrm{frank}.

\section{Performance Comparisons}
\label{sec:performance}

We now present some performance comparisons. For our tests, we form
$m \times n$ matrices of the form $\mtx{A} = \mtx{U} \mtx{D} \mtx{V}^*$,
where $\mtx{U}$ and $\mtx{V}$ are randomly drawn matrices with orthonormal columns and
where $\mtx{D}$ is a diagonal matrix with the desired singular value distribution.
We use different \textit{logspaced} distributions of singular values, which in
Matlab notation are generated using the commands
$\textrm{logspace}(0,-0.5,r)$, $\textrm{logspace}(0,-2,r)$,
and $\textrm{logspace}(0,-3.5,r)$ with $r = \min(m,n)$. We refer to the resulting matrices
as types I, II, and III.  First, we present some comparisons regarding the
accuracy of the low rank \SVD~and \ID~computations obtained with different powers of the
power scheme for the three types of matrices. In Figure \ref{fig:rsvd_diffq_errors},
we plot the approximation errors
$\frac{\|\mtx{A} - \mtx{U}_k \mtx{\Sigma}_k \mtx{V}^*_k\|}{\|\mtx{A}\|}$ vs $k$
obtained using the RSVD algorithm (with \QR~method) with different powers $q$
for the power sampling scheme. We observe that the use of the power scheme iterations ($q>1$),
gives substantially closer to optimal results than $q=0$. On the left, we present an
example where the singular
values fall off relatively slowly, and as a result, each increase in $q$ has a noticeable
effect. In particular, $q=2$ is sufficient for a good approximation and gives
better performance than $q=1$. For matrices
with rapidly decreasing singular values (type III) on the right of
Figure \ref{fig:rsvd_diffq_errors}, we observe that $q=1$ is sufficient for a
good approximation.

\begin{figure*}[ht!]
\centerline{
\includegraphics[scale=0.35]{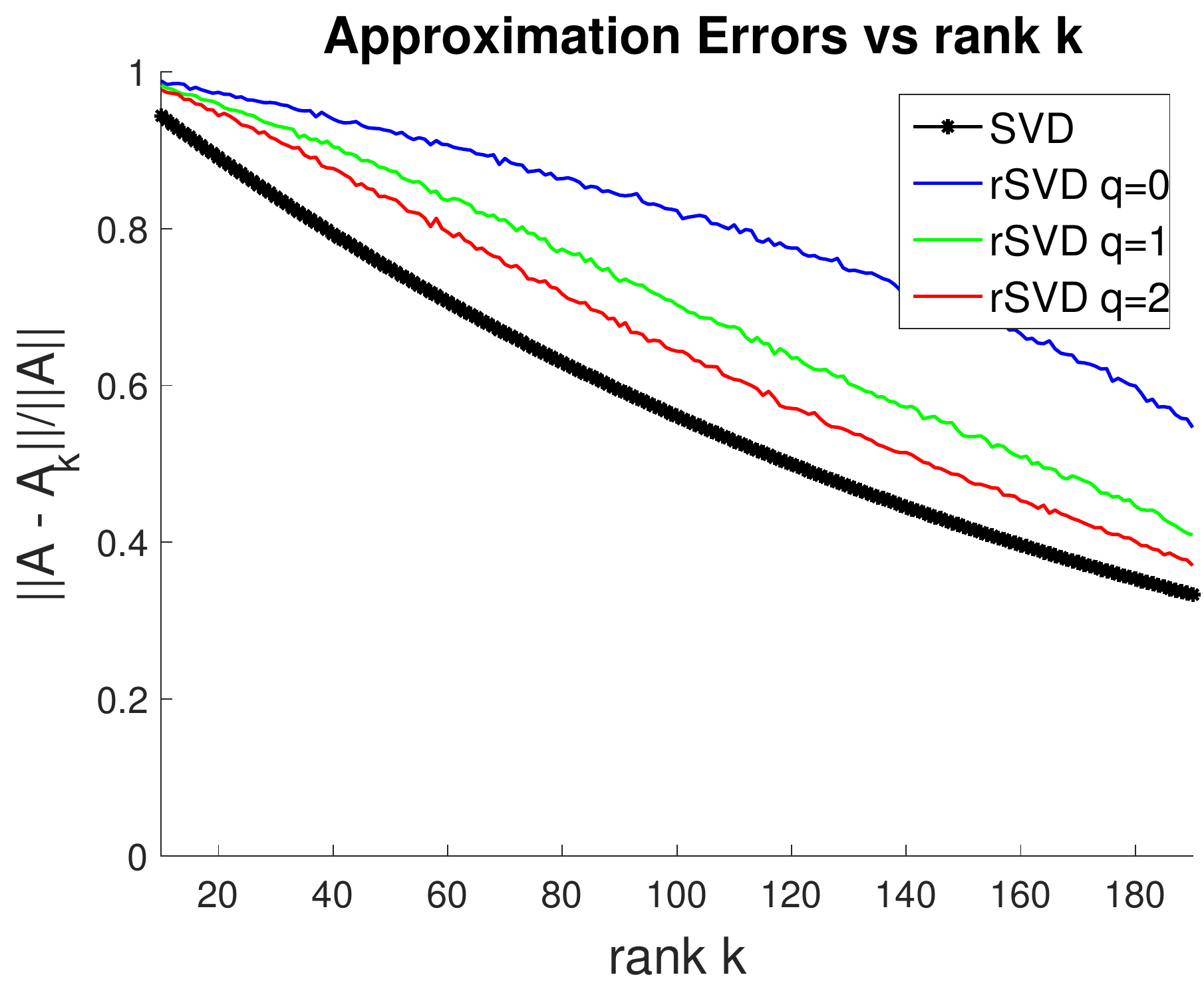}
\includegraphics[scale=0.35]{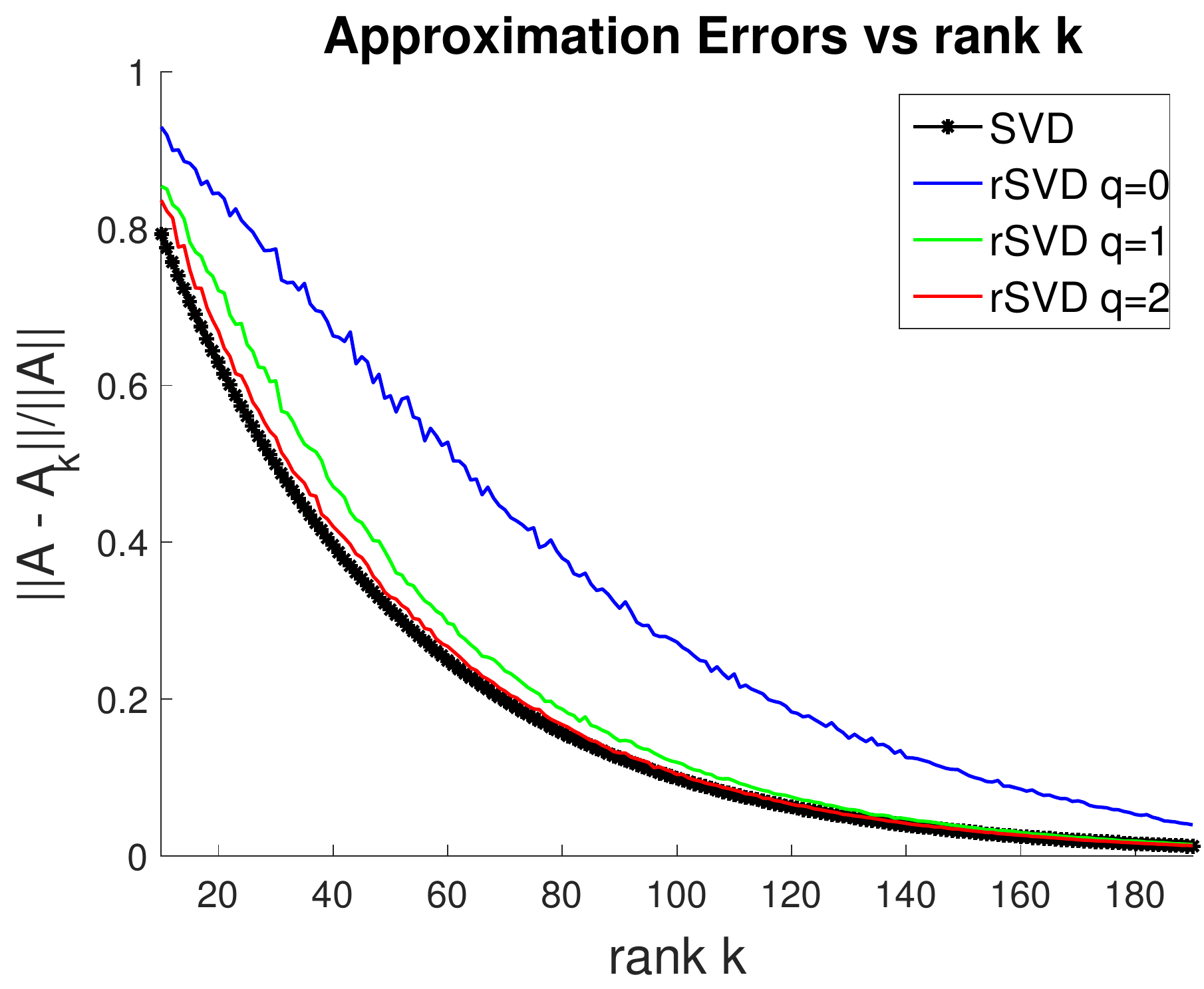}
\includegraphics[scale=0.35]{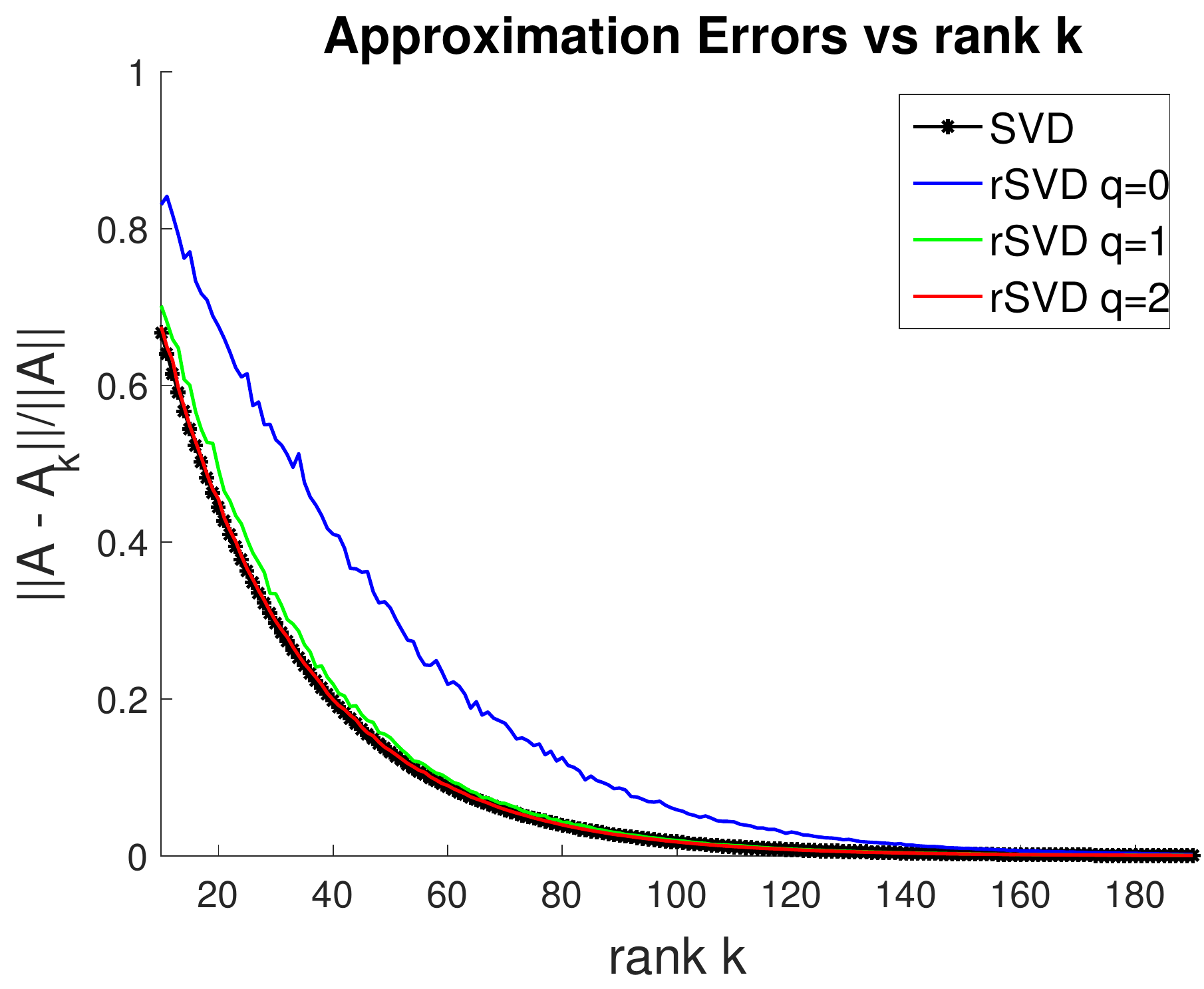}
}
\caption{Approximation errors vs rank $k$ obtained via the RSVD algorithm with
$q=0,1,2$ for better and worse conditioned matrices. \label{fig:rsvd_diffq_errors}}
\end{figure*}

Next, we compare the approximation errors we obtain for the different low rank approximation
algorithms in Figure \ref{fig:rsvd_id_and_cur_errors} for matrices of type II and III.
Recall that based on our
discussion in Section \ref{sec:matrix_decompositions}, for a fixed rank $k$, the
\QR~and single and two sided \ID~decompositions share the same error term (larger
than that of the optimal truncated \SVD), while the \CUR~adds a bit of additional error.
When we use randomized algorithms to compute the approximate decompositions, the \CUR~often
yields slightly better results than the approximate \ID.
We also observe in Figure \ref{fig:rsvd_id_and_cur_errors} that for matrices with
relatively slow singular value decay, the randomized \ID~and \CUR~algorithms give
poor approximations relative to the low rank \SVD. On the other hand, for more
rapid singular value decay, the \ID~and \CUR~approximations give better results.
\begin{figure*}[ht!]
\centerline{
\includegraphics[scale=0.35]{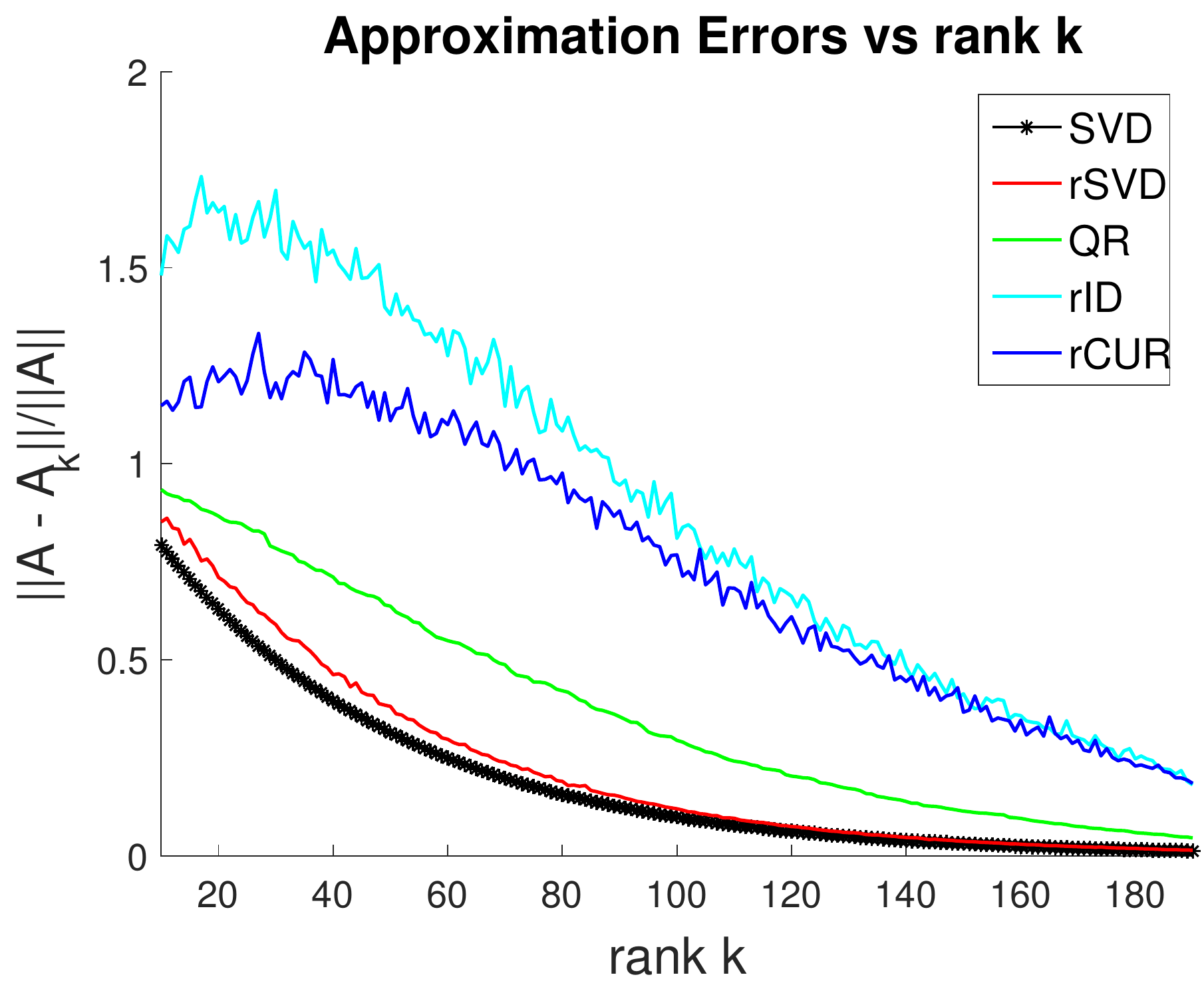}
\includegraphics[scale=0.35]{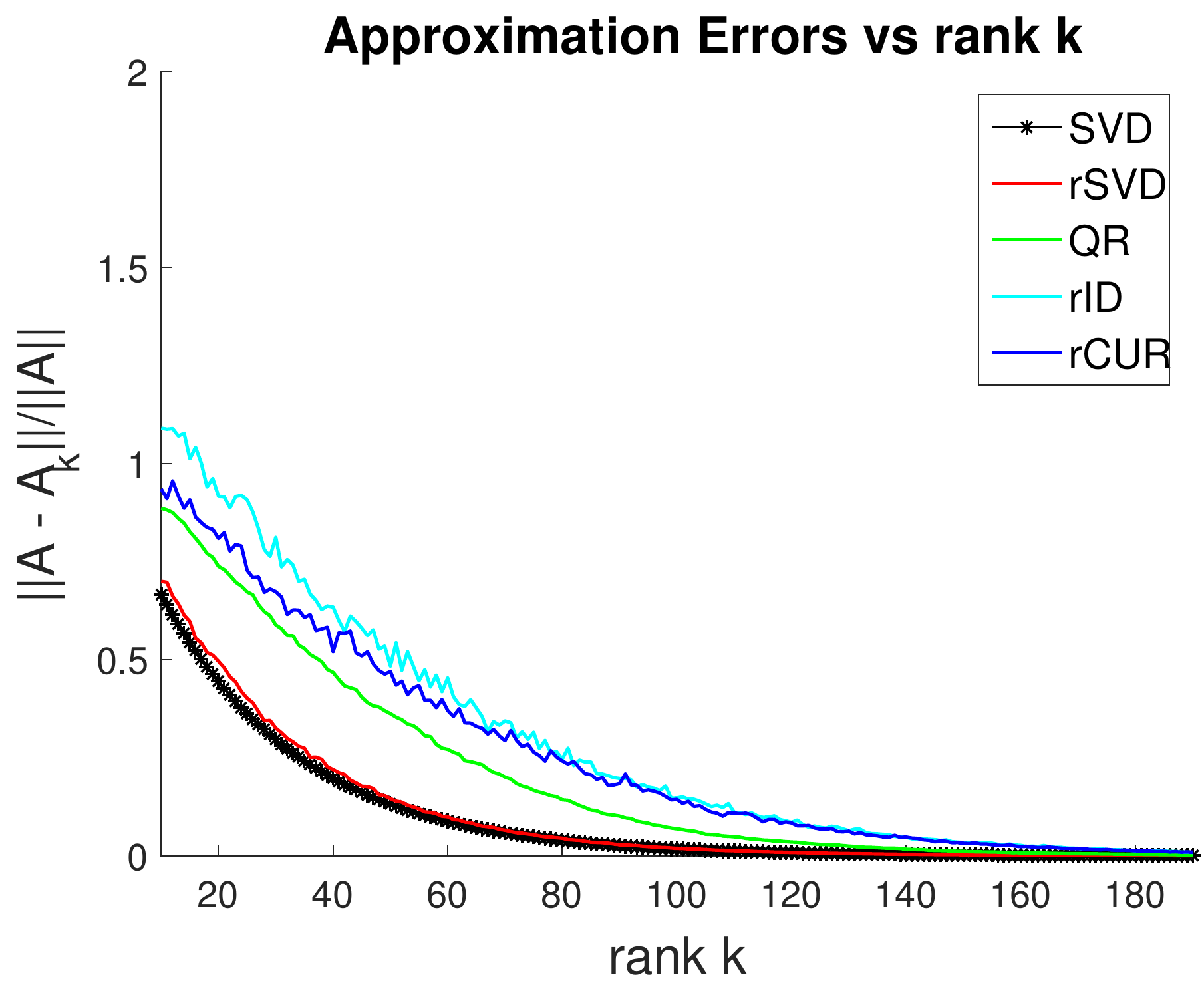}
}
\caption{Approximation errors vs rank $k$ obtained via the RSVD, RID, and RCUR algorithms with
$q=2$ for better and worse conditioned matrices. \label{fig:rsvd_id_and_cur_errors}}
\end{figure*}

However, the low rank \ID~and \CUR~approximations may 
require less storage space than the low rank \SVD. In Figure
\ref{fig:total_nnzs_diff_factorizations}, we plot the number of nonzeros
versus rank $k$ in the matrices of the low rank factorizations for a full and
sparse input matrix. When the matrix $\mtx{A}$ is not sparse,
the low rank \SVD~and the \CUR~factorizations
require the same storage space (in terms of the number of nonzeros) for a given rank.
The \ID~requires less storage because part of matrix $\mtx{V}$ in the \ID~is the $k \times k$
identity matrix. The situation changes dramatically when $\mtx{A}$ is a sparse matrix. On
the right of Figure \ref{fig:total_nnzs_diff_factorizations}, we illustrate the situation
for a $5\%$ nonzero sparse matrix. Both the \CUR~and \ID~use less nonzeros than the
low rank \SVD, for a fixed rank $k$. In fact, for the \ID, the storage size decreases as
the identity matrix occupies a greater portion of $\mtx{V}$. Hence, especially in the case 
of sparse matrices, for the same amount of memory, we can use an \ID~or~\CUR~approximation of higher rank $k$ than for an \SVD. 

\begin{figure*}[ht!]
\centerline{
\includegraphics[scale=0.35]{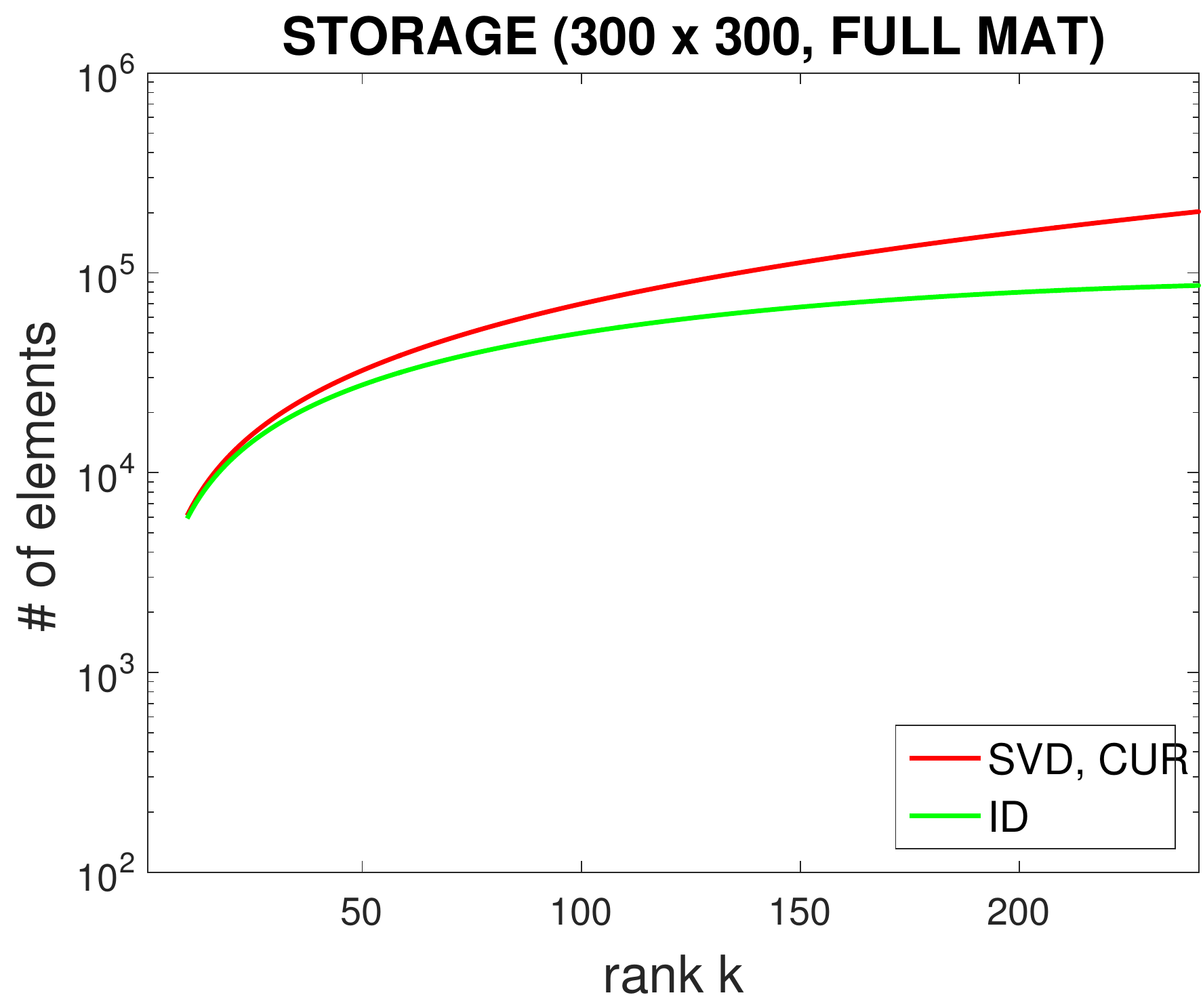}
\includegraphics[scale=0.35]{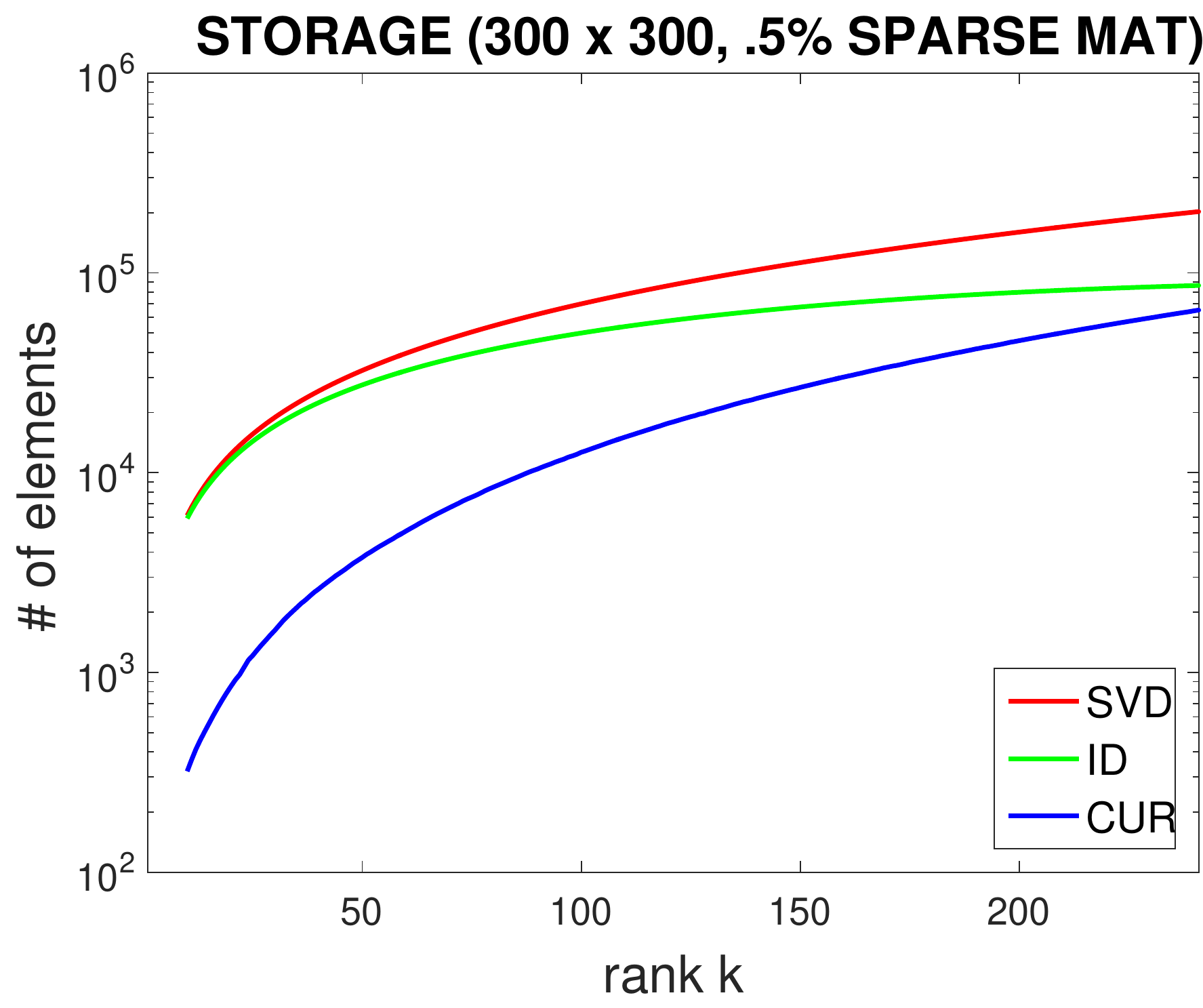}
}
\caption{Total nonzeros in different factorizations for full and sparse matrices, as a 
function of the rank.
\label{fig:total_nnzs_diff_factorizations}}
\end{figure*}

We also illustrate the runtimes of the different functions.
We perform our runs on a PC containing an
Intel Xeon E5-2440 chip (6 cores, up to $2.90$ GHz) and an NVIDIA Tesla K40c graphics card
and we use a $6000\times12000$ dense matrix (we obtain similar results for 
$12000\times6000$ matrices). In Figure \ref{fig:routine_runtimes1}, we
include the runtimes for the low rank factorizations obtained with the randomized and block
randomized algorithms and the corresponding full factorizations. Especially with a GPU, 
great time savings are obtained, since matrix-matrix multiplication on GPU is
cheap because it is very well parallelizable. The only current disadvantage of GPUs is a lack
of high memory, which puts a limit on the matrix size possible to transfer onto the card.
Note that on both architectures, the block randomized \QB~based methods are slower than
their randomized counterparts. The advantage of \QB~however, is the ability to specify
a tolerance error bound for the factorization to satisfy. One also avoids this way the
factorization and multiplication of large matrices, which can be an advantage for
a large $\mtx{A}$. 

In Figure \ref{fig:routine_runtimes1}, we also plot the runtimes of
the full \SVD, full \QR, and an estimated lower bound for a partial pivoted 
\QR~factorization (this routine is implemented in RSVDPACK, but is currently 
slower than optimal). Notice that both the full \SVD~and
\QR~are too expensive to compute if only a low rank factorization is desired. On the other
hand the partial \QR~(obtained by doing only $k$ iterations of Gram-Schmidt) is somewhat
competitive on the CPU, although the resulting approximation error bound is of course higher
than that of the low rank \SVD. On the GPU, however, all randomized schemes show
significantly lower runtimes. This is clearly explained by the difference in
matrix matrix multiplication times plotted in row two of Figure
\ref{fig:routine_runtimes1}. Since our algorithms are heavily based on matrix
matrix multiplications, they perform well on the GPU. Also in Figure \ref{fig:routine_runtimes1},
we compare the runtimes of the Matlab mex interface routines we include for 
$n\times n$ matrices (of type II) using ranks $k = 100,300,500$ to the PROPACK 
\cite{larsen2004propack} package partial \SVD~library function \textrm{lansvd}, 
also partially accelerated by means of mex file 
subroutines. We observe a significant speedup with our mex interfaced routine.

\newpage
\begin{figure*}[ht!]
\centerline{
\includegraphics[scale=0.38]{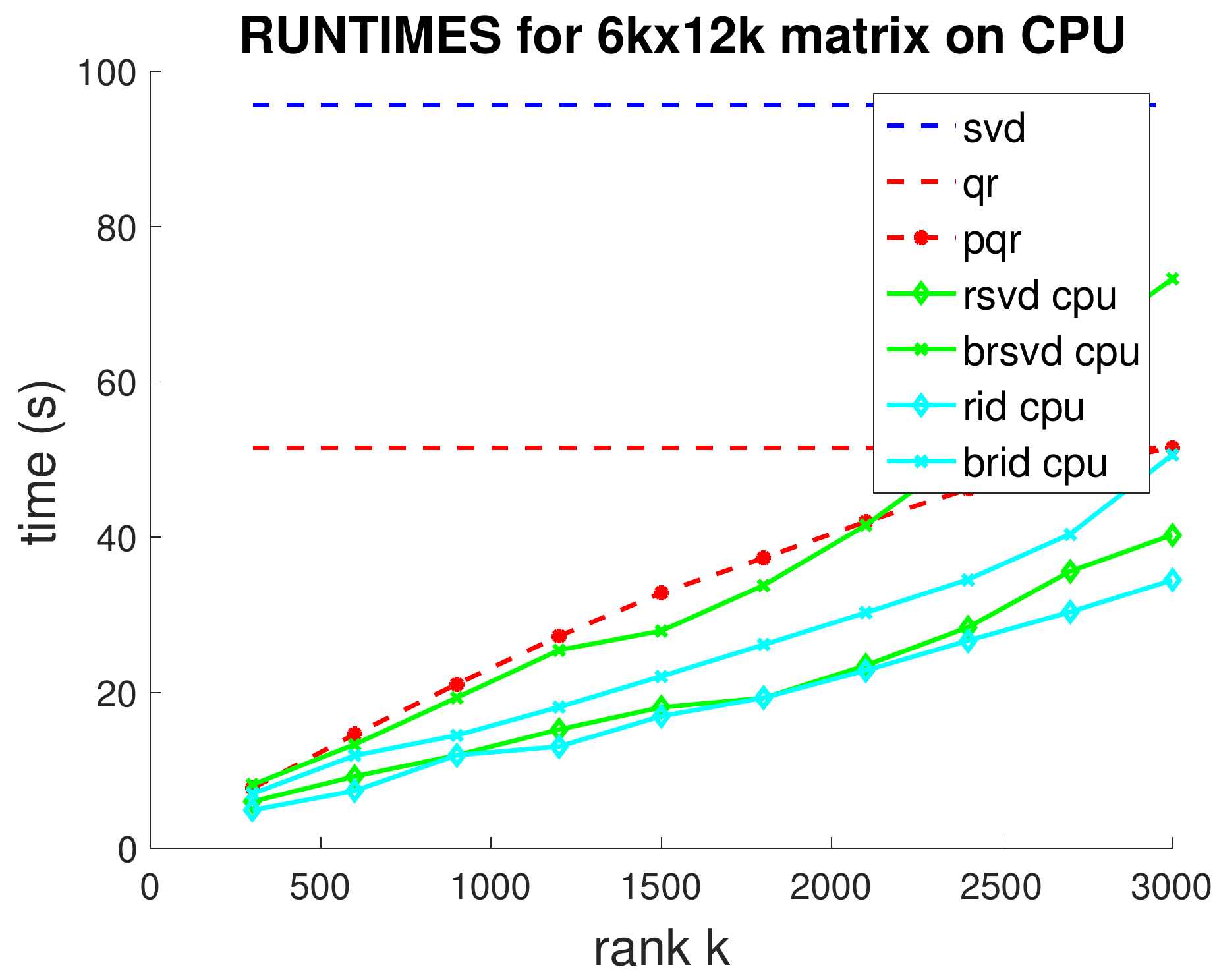}
\includegraphics[scale=0.38]{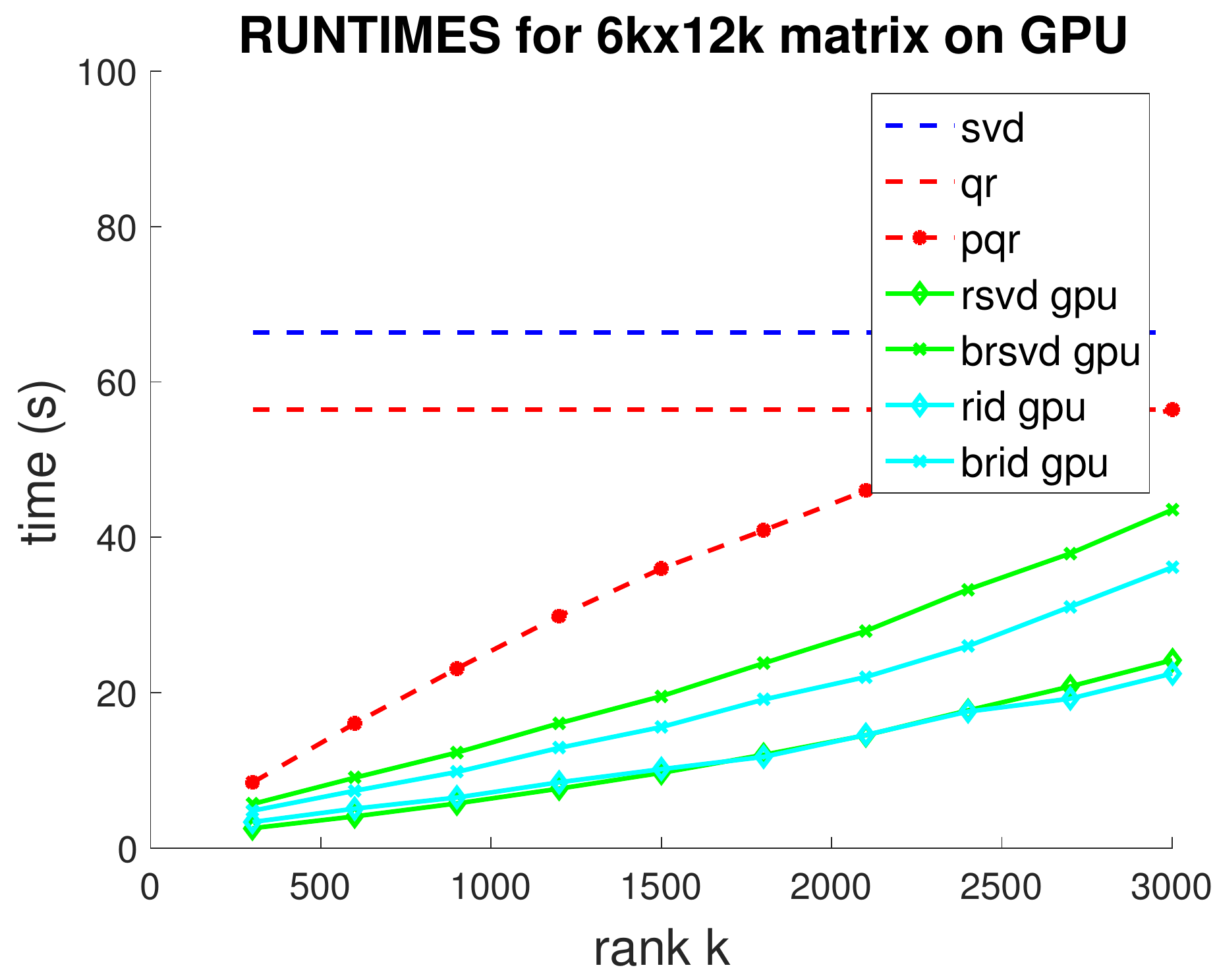}
}
\centerline{
\includegraphics[scale=0.38]{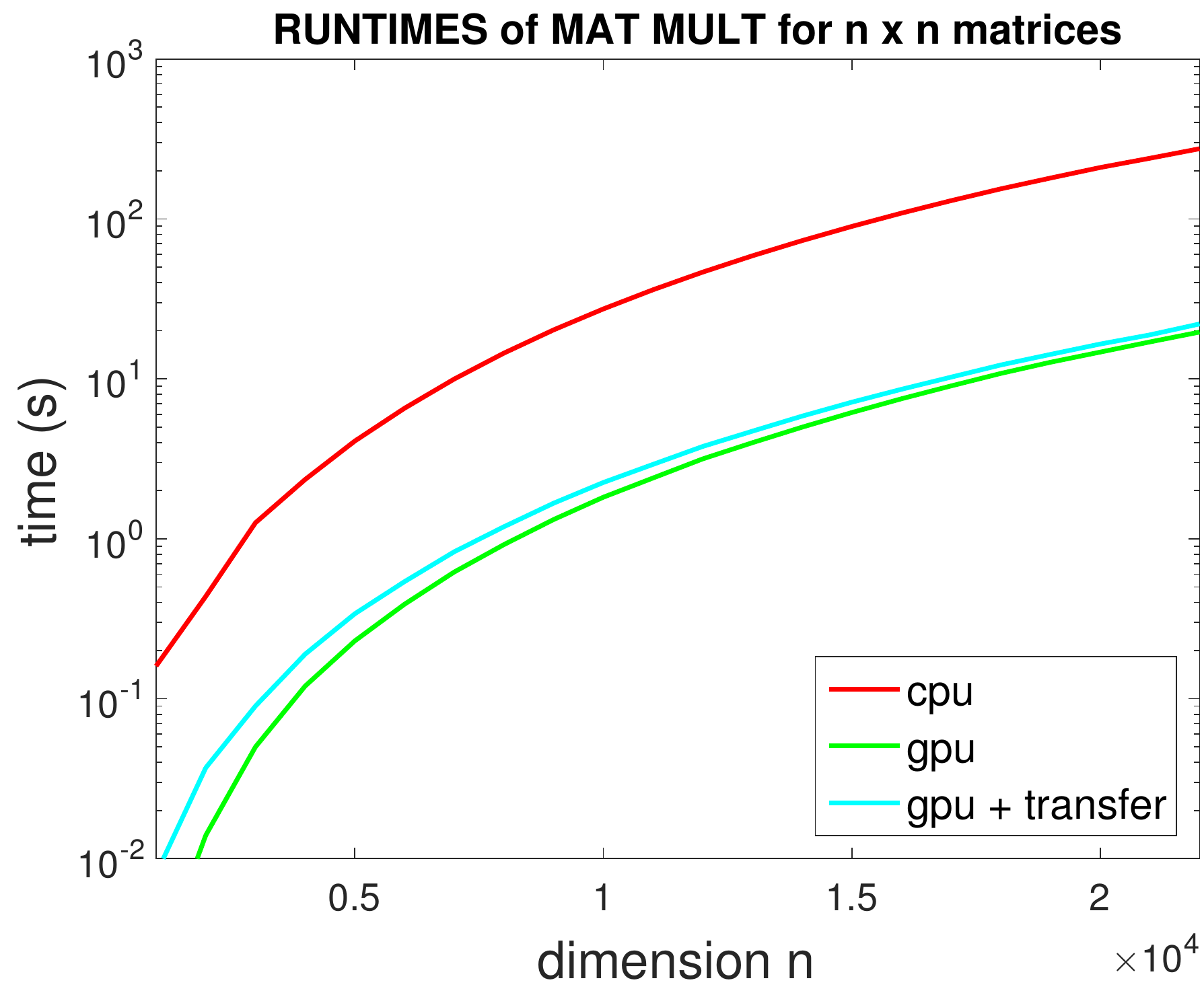}
\includegraphics[scale=0.38]{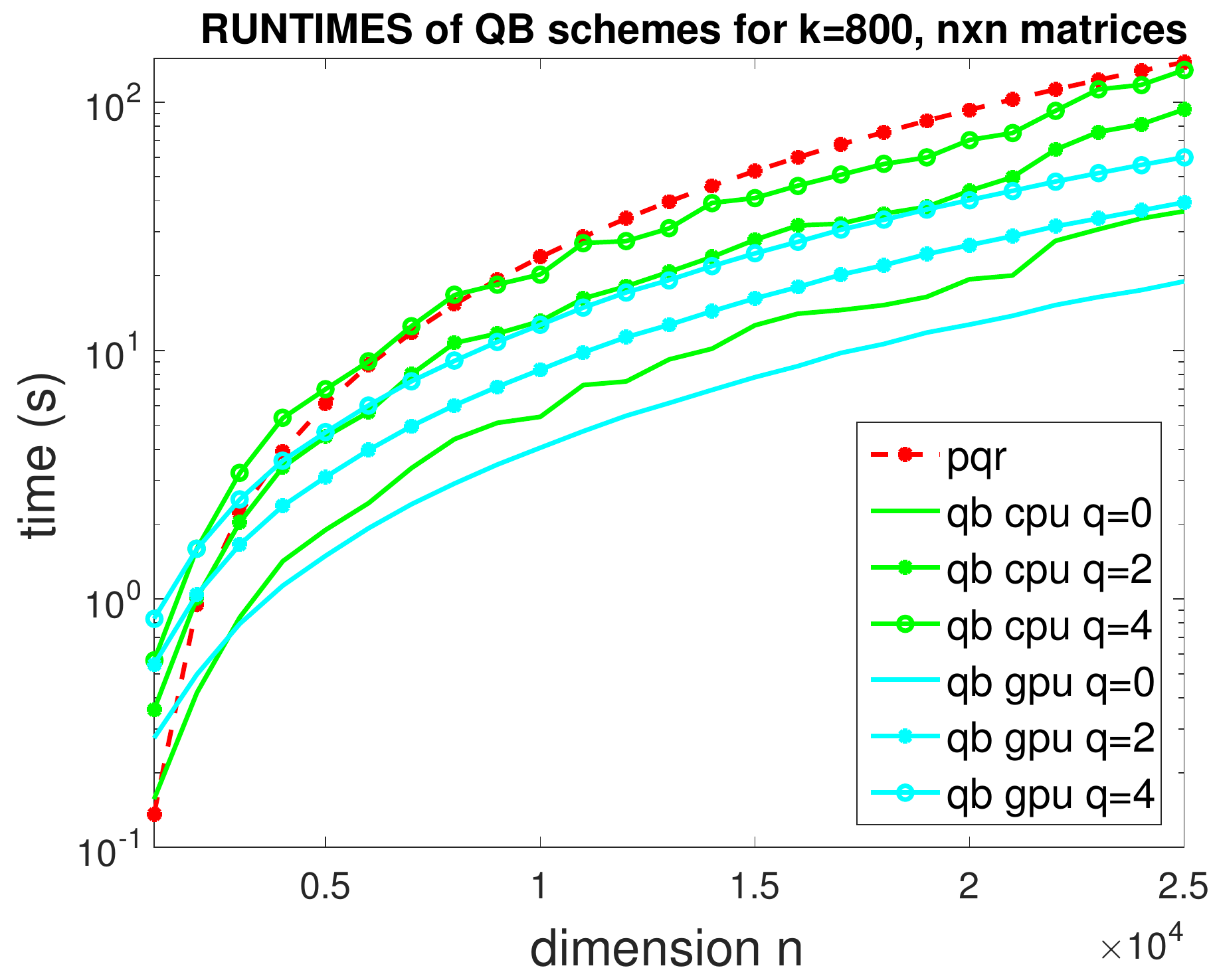}
}
\centerline{
\includegraphics[scale=0.38]{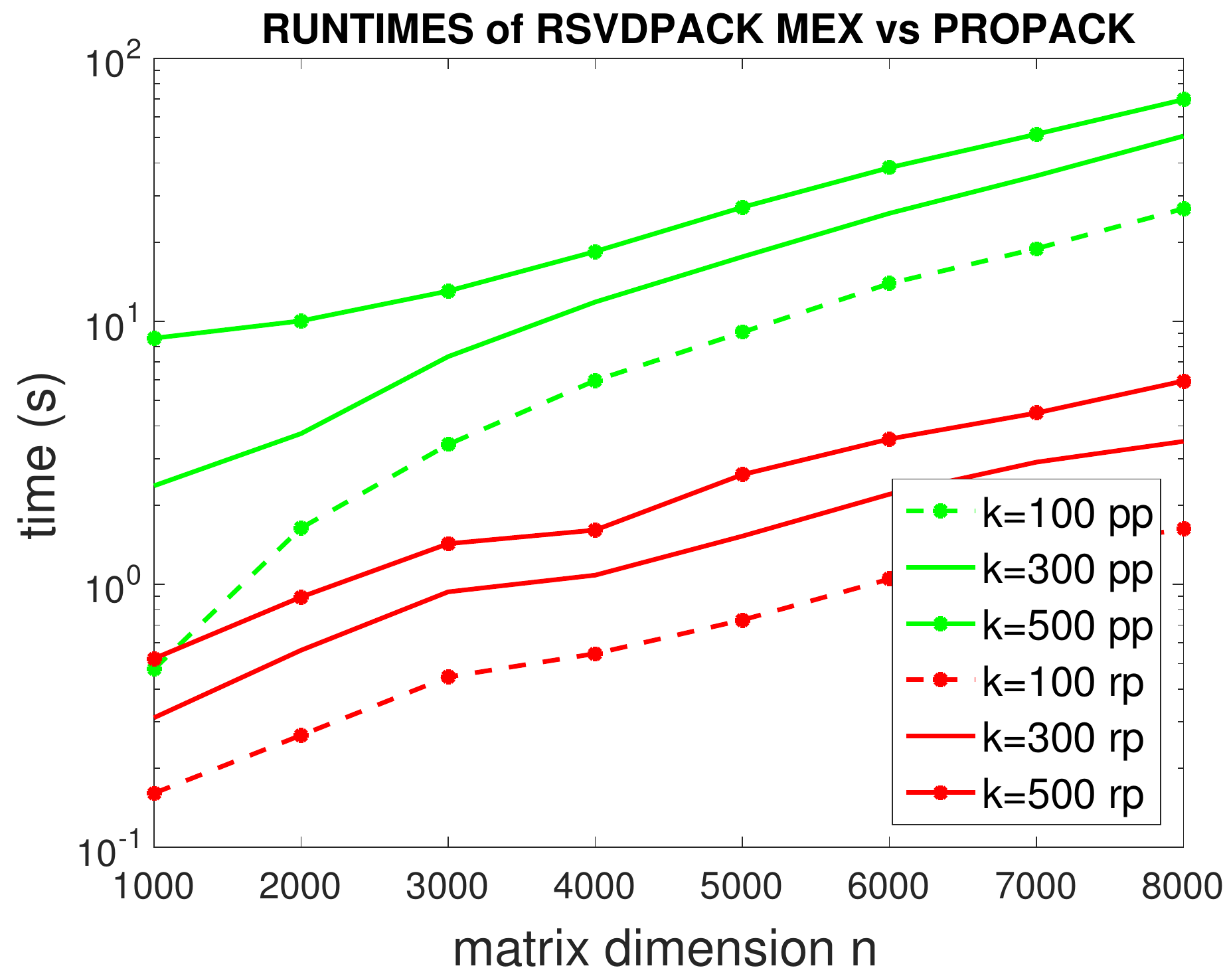}
}
\caption{First row: runtimes of \SVD, \QR, partial \QR, and
randomized and block randomized \SVD~and
\ID~factorizations for a $6000\times12000$ matrix on CPU and GPU. Second row:
runtimes of matrix matrix multiplication on CPU and GPU (and GPU with memory transfer from RAM)
for $n \times n$ matrices. Third row: runtimes of rsvdpack (rp) mex interface (with MKL) for low rank \SVD~of $n \times n$ matrices with different $k$ vs propack (pp) lansvd mex accelerated algorithm. \label{fig:routine_runtimes1}}
\end{figure*}

\newpage

\section{Conclusions}
\label{sec:conc}

This article presents the mathematical details of RSVDPACK: an open source software package
for efficiently computing low rank \SVD, \ID, and \CUR~factorizations of matrices. The package
currently provides the end user with functions to perform each factorization
using a non-randomized, randomized, or block randomized algorithm using an input
rank $k$ or a tolerance value. The package is divided into several C codes supporting 
both multi-core and GPU architectures and provides a Matlab mex file interface.
The provided routines should be suitable for a wide range of applications
and the package can be easily modified to match the necessary input/output formats.
The software package we present is frequently being updated and enhanced with new
functionality. Currently, we are further extending our set of accelerated routines for
Matlab by means of mex files and planning for the development of routines targeting
very large matrices using distributed memory computation. We also plan to add 
support for sparse matrices and further improve performance of the routines on GPUs.

\section{Availability of software}
\label{sec:availability}

The latest version of the open source software can be obtained from
\url{https://github.com/sergeyvoronin} and is available with a GNU GPL v3 license.
The multi core implementation relies on the Intel MKL library available from
\url{https://software.intel.com/en-us/intel-mkl}. The GPU accelerated implementations
rely on MKL and the CULA dense library available from \url{http://www.culatools.com/},
built atop NVIDIA's CUDA framework or on MKL and NVIDIA cuBLAS (available with CUDA 
from \url{https://developer.nvidia.com/cuda-zone}). 
The MKL and CULA libraries are not open source, but are available under a variety of 
licensing terms.

\vspace{4mm}

\textbf{Acknowledgments:}
The research reported was supported by DARPA, under the contract N66001-13-1-4050,
by the NSF, under the contract DMS-1320652, and by an equipment award from NVIDIA.
\textit{Any opinions, findings and conclusions or recommendations
 expressed in this material are those of the author(s) and do not
 necessarily reflect the views of the National Science Foundation
 (NSF).}

\bibliographystyle{plain}
\bibliography{main_bib}

\begin{thebibliography}{10}

\bibitem{boutsidis2014optimal}
Christos Boutsidis and David~P Woodruff.
\newblock Optimal cur matrix decompositions.
\newblock In {\em Proceedings of the 46th Annual ACM Symposium on Theory of
  Computing}, pages 353--362. ACM, 2014.

\bibitem{2005_martinsson_skel}
H.~Cheng, Z.~Gimbutas, P.G. Martinsson, and V.~Rokhlin.
\newblock On the compression of low rank matrices.
\newblock {\em SIAM Journal of Scientific Computing}, 26(4):1389--1404, 2005.

\bibitem{duersch2015true}
Jed~A Duersch and Ming Gu.
\newblock True blas-3 performance qrcp using random sampling.
\newblock {\em arXiv preprint arXiv:1509.06820}, 2015.

\bibitem{eckart1936approximation}
Carl Eckart and Gale Young.
\newblock The approximation of one matrix by another of lower rank.
\newblock {\em Psychometrika}, 1(3):211--218, 1936.

\bibitem{Erichson16}
N.~Benjamin Erichson.
\newblock rsvd.
\newblock \url{https://github.com/Benli11/rSVD}, 2016.

\bibitem{golub}
Gene~H. Golub and Charles~F. Van~Loan.
\newblock {\em Matrix computations}.
\newblock Johns Hopkins Studies in the Mathematical Sciences. Johns Hopkins
  University Press, Baltimore, MD, third edition, 1996.

\bibitem{gu1996}
Ming Gu and Stanley~C. Eisenstat.
\newblock Efficient algorithms for computing a strong rank-revealing {QR}
  factorization.
\newblock {\em SIAM J. Sci. Comput.}, 17(4):848--869, 1996.

\bibitem{2011_martinsson_randomsurvey}
Nathan Halko, Per-Gunnar Martinsson, and Joel~A. Tropp.
\newblock Finding structure with randomness: Probabilistic algorithms for
  constructing approximate matrix decompositions.
\newblock {\em SIAM Review}, 53(2):217--288, 2011.

\bibitem{fbpca}
Facebook Inc.
\newblock fbpca.
\newblock \url{https://github.com/facebook/fbpca}, 2016.

\bibitem{1966_kahan_NLA}
William Kahan.
\newblock Numerical linear algebra.
\newblock {\em Canadian Math. Bull}, 9(6):757--801, 1966.

\bibitem{larsen2004propack}
Rasmus~Munk Larsen.
\newblock Propack-software for large and sparse svd calculations.
\newblock {\em Available online. URL http://sun. stanford. edu/rmunk/PROPACK},
  pages 2008--2009, 2004.

\bibitem{2007_martinsson_PNAS}
Edo Liberty, Franco Woolfe, Per-Gunnar Martinsson, Vladimir Rokhlin, and Mark
  Tygert.
\newblock Randomized algorithms for the low-rank approximation of matrices.
\newblock {\em Proc. Natl. Acad. Sci. USA}, 104(51):20167--20172, 2007.

\bibitem{2009_mahoney_CUR}
Michael~W Mahoney and Petros Drineas.
\newblock Cur matrix decompositions for improved data analysis.
\newblock {\em Proceedings of the National Academy of Sciences},
  106(3):697--702, 2009.

\bibitem{2015arXiv150307157M}
P.-G. {Martinsson} and S.~{Voronin}.
\newblock {A randomized blocked algorithm for efficiently computing
  rank-revealing factorizations of matrices}.
\newblock {\em To appear in SIAM Journal on Scientific Computation}, March
  2015.

\bibitem{2006_martinsson_random1_orig}
Per-Gunnar Martinsson, Vladimir Rokhlin, and Mark Tygert.
\newblock A randomized algorithm for the approximation of matrices.
\newblock Technical Report Yale CS research report YALEU/DCS/RR-1361, Yale
  University, Computer Science Department, 2006.

\bibitem{martinsson2008id}
PG~Martinsson, V~Rokhlin, Y~Shkolnisky, and M~Tygert.
\newblock Id: a software package for low-rank approximation of matrices via
  interpolative decompositions, 2008.

\bibitem{sorensen2016deim}
Danny~C Sorensen and Mark Embree.
\newblock A deim induced cur factorization.
\newblock {\em SIAM Journal on Scientific Computing}, 38(3):A1454--A1482, 2016.

\bibitem{2014arXiv1412.8447V}
S.~{Voronin} and P.-G. {Martinsson}.
\newblock {Efficient Algorithms for CUR and Interpolative Matrix
  Decompositions}.
\newblock {\em ArXiv e-prints}, December 2014.

\end{thebibliography}

\newpage

\begin{appendix}

\section{Summary of algorithms}
\label{app:pseudo}
In this section, we present the pseudocode for the different algorithms 
discussed in the text. We first present the \ID~(Algorithm \ref{algo:rank_k_id}), 
the two-sided \ID~(Algorithm \ref{algo:rank_k_tsid}), and the 
\CUR~decomposition~(Algorithm \ref{algo:rank_k_cur}) without the use of randomization, as they were discussed in section \ref{sec:matrix_decompositions}. Notice that 
all of these algorithms rely on the rank $k$ pivoted \QR~factorization, which is 
usually not provided as a built in function in existing LAPACK packages. 
The low rank \SVD~of rank $k$ is trivial to obtain without randomization, 
from the full \SVD. 

\begin{figure*}[!ht]
\centering
\begin{minipage}[c]{15.5cm}
\begin{algorithm}[H]
\SetKwInOut{Input}{Input}
\SetKwInOut{Output}{Output}
\caption{A rank $k$ \ID~decomposition \label{algo:rank_k_id}}
\Input{$\mtx{A}\in\mathbb{C}^{m\times n}$ and parameter $k < \min(m,n)$.}
\Output{\mtx{A} column index set $J$ and a matrix
$\mtx{V}\in\mathbb{C}^{n\times k}$
such that $\mtx{A} \approx \mtx{A}(:,J(1:k)) \mtx{V}^{*}$.}
\BlankLine
Perform a rank $k$ column pivoted \QR~factorization to get
$\mtx{A} \mtx{P} \approx \mtx{Q}_1 \mtx{S}_1$\;
\Indp
\Indm define the ordered index set $J$ via $\mtx{I}(:,J) = \mtx{P}$\;
partition $\mtx{S}_1$: $\mtx{S}_{11} = \mtx{S}_1(:,1:k)$,
$\mtx{S}_{12} = \mtx{S}_1(:,k+1:n)$\;
$\mtx{V} =
\mtx{P} \begin{bmatrix} \mtx{I}_k & \mtx{S}_{11}^{-1} \mtx{S}_{12} \end{bmatrix}^*$\;
\end{algorithm}
\end{minipage}
\end{figure*}

\begin{figure*}[!ht]
\centering
\begin{minipage}[c]{15.5cm}
\begin{algorithm}[H]
\SetKwInOut{Input}{Input}
\SetKwInOut{Output}{Output}
\caption{A rank $k$ two sided \ID~decomposition \label{algo:rank_k_tsid}}
\Input{$\mtx{A}\in\mathbb{C}^{m\times n}$ and parameter $k < \min(m,n)$.}
\Output{A column index set $J$, a row index set $I$ and a matrices
$\mtx{V}\in\mathbb{C}^{n\times k}$ and $\mtx{W}\in\mathbb{C}^{m \times k}$
such that $\mtx{A} \approx \mtx{W} \mtx{A}(I(1:k),J(1:k)) \mtx{V}^{*}$.}
\BlankLine
Perform a one sided rank $k$ \ID~of $\mtx{A}$ so that
$\mtx{A} \approx \mtx{C} \mtx{V}^{*}$ where $\mtx{C} = \mtx{A}(:,J(1:k))$\;
Perform a full rank \ID~on $\mtx{C}^*$ so that
$\mtx{C}^{*} = \mtx{C}^*(:,I(1:k)) \mtx{W}^{*}$\;
\end{algorithm}
\end{minipage}
\end{figure*}

\begin{figure*}[!ht]
\centering
\begin{minipage}[c]{15.5cm}
\begin{algorithm}[H]
\SetKwInOut{Input}{Input}
\SetKwInOut{Output}{Output}
\caption{A rank $k$ \CURID~algorithm \label{algo:rank_k_cur}}
\Input{$\mtx{A}\in\mathbb{C}^{m\times n}$ and parameter $k < \min(m,n)$.}
\Output{Matrices $\mtx{C} \in \mathbb{C}^{m \times k}$, $\mtx{R} \in \mathbb{C}^{k \times n}$, and
$\mtx{U} \in \mathbb{C}^{k \times k}$ (such that $\mtx{A} \approx \mtx{C} \mtx{U} \mtx{R}$).}
\BlankLine
Construct a rank $k$ two sided \ID~of $\mtx{A}$ so that
$\mtx{A} \approx \mtx{W} \mtx{A}(I(1:k),J(1:k)) \mtx{V}^{*}$\;
Construct matrices $\mtx{C} = \mtx{A}(:,J(1:k))$ and $\mtx{R} = \mtx{A}(I(1:k),:)$\;
Construct matrix $\mtx{U}$ via $\mtx{U} = \mtx{V}^{*} \mtx{R}^{\dagger}$\;
\end{algorithm}
\end{minipage}
\end{figure*}

Next, we present the randomized algorithms for computing the approximate low 
rank \SVD, the \ID, and \QB~decompositions. First we present the 
two low rank \SVD~methods from section
\ref{sec:randomized_algorithms_for_low_rank_svd}. 
Algorithm \ref{algo:rank_rsvd1}, which uses the eigendecomposition 
of the small $\mtx{B} \mtx{B}^*$ matrix in place of the \SVD~of $\mtx{B}$, 
and Algorithm \ref{algo:rank_rsvd2} which uses a \QR~factorization of 
$\mtx{B}^*$ to construct a small matrix $\mtx{\hat{R}}$ on which the \SVD~is 
performed. Notice that for very large matrices $\mtx{A}$, even the 
corresponding smaller matrix $\mtx{B}$ (which will in general not be sparse 
even if $\mtx{A}$ is) may still be too large to multiply. A derivative of 
Algorithm  \ref{algo:rank_rsvd1} can be used in this case. The matrix matrix 
product $\mtx{B} \mtx{B}^*$ can be evaluated a column at a time 
(via multiplication with standard basis vectors, as in 
$\mtx{B} \mtx{B}^* \vct{e}_j$) when $\mtx{B}$ is too large. The 
same can be done for the computation of $\mtx{V}_k$ via the matrix-matrix 
product $\mtx{B}^{*} \hat{\mtx{U}} \mtx{\Sigma}_k^{-1}$.

Notice that for both Algorithms \ref{algo:rank_rsvd1} and \ref{algo:rank_rsvd2}, 
the largest (by magnitude) $k$ singular value components are extracted 
at the end of the procedure.  However, in some software packages (like 
Matlab) the eigenvalue ordering for the eigendecomposition is opposite to that 
of the singular value ordering for the \SVD. For this reason, the last 
line of Algorithm \ref{algo:rank_rsvd1} shows the last $k$ components 
being extracted. Notice that steps $14$ and $15$ of both 
methods can be combined to yield more efficient computations with 
smaller matrices (e.g. $\mtx{U}_k = \mtx{Q} \hat{\mtx{U}}(:,(p+1):l)$).  

Next, we show the pseudocode for the randomized \ID~decomposition based on 
the discussion in section \ref{sec:randidalgs}. 
Algorithm \ref{algo:rank_k_rand_id} can be further enhanced by using 
the power sampling scheme, as discussed in \ref{sec:randidalgs}.

Finally, in Algorithms \ref{algo:randpbQB2} and \ref{algo:randpbQB3}, we 
present two randomized variants of the \QB~decomposition discussed in section 
\ref{sec:randomized_algorithms}. Of these, Algorithm \ref{algo:randpbQB3} 
is only approximate, but has the advantage that it can be 
further parallelized than Algorithm \ref{algo:randpbQB2} and does not require the 
update of the original (or copy of) matrix $\mtx{A}$.

\begin{figure*}[!ht]
\centering
\begin{minipage}[c]{14cm}
\begin{algorithm}[H]
\SetKwInOut{Input}{Input}
\SetKwInOut{Output}{Output}
\caption{RSVD Algorithm Version I \label{algo:rank_rsvd1}}
\Input{$\mtx{A}\in\mathbb{R}^{m\times n}$, integer rank parameter $k < \min(m,n)$,
an integer oversampling parameter $p>0$, an integer power sampling parameter $q \geq 0$,
and an integer re-orthogonalization amount parameter $s \geq 1$.}
\Output{Matrices $\mtx{U}_k \in \mathbb{R}^{m \times k}$, $\mtx{\Sigma}_k \in \mathbb{R}^{k \times k}$, and $\mtx{V}_k \in \mathbb{R}^{k \times n}$.}
\BlankLine
Set $l = k+p$ and initialize a matrix $\mtx{R} \in \mathbb{R}^{n \times l}$ with Gaussian random entries\;
Form samples matrix $\mtx{Y} = \mtx{A} \mtx{R}$ and utilize optional power scheme\;
\For{$j=1$ \KwTo $q$}{
\If{$\mod{\left((2j-2),s\right)} == 0$}{
$[\mtx{Y},\cdot] = \qr(\mtx{Y},0)$\;
}
$\mtx{Z} = \mtx{A}^{*} \mtx{Y}$\;
\If{$\mod{\left((2j-1),s\right)} == 0$}{
$[\mtx{Z},\cdot] = \qr(\mtx{Z},0)$\;
}
$\mtx{Y} = \mtx{A} \mtx{Z}$\;
}
Orthonormalize the columns of $\mtx{Y}$ in $[\mtx{Q},\cdot] = \qr(\mtx{Y},0)$\;

Obtain the smaller matrix $\mtx{B} = \mtx{Q}^{*} \mtx{A}$ derived from $\mtx{A}$\;
Form the even smaller $l \times l$ matrix $\mtx{T} = \mtx{B} \mtx{B}^{*}$\;
Perform eigendecomposition of $l \times l$ matrix $[\hat{\mtx{U}},\mtx{D}] = \textrm{eig}(\mtx{T})$ \;
Form the approximate low rank \SVD~components of $\mtx{A}$ using the results of the eigendecomposition $\mtx{\Sigma}_k = \sqrt{\mtx{D}}$, $\mtx{U}_k = \mtx{Q} \hat{\mtx{U}}$, $\mtx{V}_k = \mtx{B}^{*} \hat{\mtx{U}} \mtx{\Sigma}_k^{-1}$\;
Extract components corresponding to the $k$ largest by magnitude singular values 
$\mtx{U}_k = \mtx{U}_k(:,(p+1):l); \quad \mtx{\Sigma}_k = \mtx{\Sigma}_k((p+1):l,(p+1):l); \quad \mtx{V}_k = \mtx{V}_k(:,(p+1):l)$\;
\end{algorithm}
\end{minipage}
\end{figure*}

\newpage

\begin{figure*}[!ht]
\centering
\begin{minipage}[c]{14cm}
\begin{algorithm}[H]
\SetKwInOut{Input}{Input}
\SetKwInOut{Output}{Output}
\caption{RSVD Algorithm Version II \label{algo:rank_rsvd2}}
\Input{$\mtx{A}\in\mathbb{R}^{m\times n}$, integer rank parameter $k < \min(m,n)$,
an integer oversampling parameter $p>0$, an integer power sampling parameter $q \geq 0$,
and an integer re-orthogonalization amount parameter $s \geq 1$.}
\Output{Matrices $\mtx{U}_k \in \mathbb{R}^{m \times k}$, $\mtx{\Sigma}_k \in \mathbb{R}^{k \times k}$,
and $\mtx{V}_k \in \mathbb{R}^{k \times n}$.}
\BlankLine
Set $l = k+p$ and initialize a matrix $\mtx{R} \in \mathbb{R}^{n \times l}$ with Gaussian random entries\;
Form samples matrix $\mtx{Y} = \mtx{A} \mtx{R}$ and utilize optional power scheme\;
\For{$j=1$ \KwTo $q$}{
\If{$\mod{\left((2j-2),s\right)} == 0$}{
$[\mtx{Y},\cdot] = \qr(\mtx{Y},0)$\;
}
$\mtx{Z} = \mtx{A}^{*} \mtx{Y}$\;
\If{$\mod{\left((2j-1),s\right)} == 0$}{
$[\mtx{Z},\cdot] = \qr(\mtx{Z},0)$\;
}
$\mtx{Y} = \mtx{A} \mtx{Z}$\;
}
Orthonormalize the columns of $\mtx{Y}$ in $[\mtx{Q},\cdot] = \qr(\mtx{Y},0)$\;
Compute the smaller matrix $\mtx{Bt} = \mtx{A}^{*} \mtx{Q}$\;
Obtain the small $l \times l$ matrix $\hat{\mtx{R}}$ using a compact \QR~factorization $[\hat{\mtx{Q}},\hat{\mtx{R}}] = \qr(\mtx{Bt},0)$\;
Take the \SVD~of the $l \times l$ matrix $\hat{\mtx{R}}$, $[\hat{\mtx{U}},\mtx{\Sigma}_k,\hat{\mtx{V}}] = \svd(\hat{\mtx{R}})$\;
Form the approximate low rank \SVD~components of $\mtx{A}$ using the results of the \SVD~of 
$\hat{\mtx{R}}$. 
$\mtx{U}_k = \mtx{Q} \hat{\mtx{V}}$, 
$\mtx{V}_k = \hat{\mtx{Q}} \hat{\mtx{U}}$\;
Extract components corresponding to the $k$ largest by magnitude singular values 
$\mtx{U}_k = \mtx{U}_k(:,1:k); \quad \mtx{\Sigma}_k = \mtx{\Sigma}_k(1:k,1:k); \quad \mtx{V}_k = \mtx{V}_k(:,1:k)$\;
\end{algorithm}
\end{minipage}
\end{figure*}

\begin{figure*}[!ht]
\centering
\begin{minipage}[c]{15.8cm}
\begin{algorithm}[H]
\SetKwInOut{Input}{Input}
\SetKwInOut{Output}{Output}
\caption{A randomized rank $k$ \ID~decomposition \label{algo:rank_k_rand_id}}
\Input{$\mtx{A}\in\mathbb{R}^{m\times n}$, a rank parameter $k < \min(m,n)$,
and an oversampling parameter $p$.}
\Output{A column index set $J$ and a matrix
$\mtx{V}\in\mathbb{R}^{n\times k}$
(such that $\mtx{A} \approx \mtx{A}(:,J(1:k)) \mtx{V}^{*}$).}
\BlankLine
Construct a random matrix $\mtx{\Omega} \in \mathbb{R}^{(k+p) \times m}$
with i.i.d.~Gaussian entries\;
Construct the sample matrix $\mtx{Y} = \mtx{\Omega} \mtx{A}$\;
Perform full pivoted \QR~factorization on $\mtx{Y}$ to get:
$\mtx{Y} \mtx{P} = \mtx{Q} \mtx{S}$\;
Remove $p$ columns of $\mtx{Q}$ and $p$ rows of $\mtx{S}$ to construct
$\mtx{Q}_1$ and $\mtx{S}_1$\;
Define the ordered index set $J$ via $\mtx{I}(:,J) = \mtx{P}$\;
Partition $\mtx{S}_1$: $\mtx{S}_{11} = \mtx{S}_1(:,1:k)$,
$\mtx{S}_{12} = \mtx{S}_1(:,k+1:n)$\;
$\mtx{V} = \mtx{P} \begin{bmatrix} \mtx{I}_k & \mtx{S}_{11}^{-1} \mtx{S}_{12} \end{bmatrix}^*$\;
\end{algorithm}
\end{minipage}
\end{figure*}

\vspace{3.mm}

\begin{figure}
\begin{algorithm}[H]
\SetKwInOut{Input}{Input}
\SetKwInOut{Output}{Output}
\caption{A randomized blocked \QB~decomposition \label{algo:randpbQB2}}
\Input{$\mtx{A}\in\mathbb{R}^{m\times n}$, integer block size parameter $b$,
maximum number of blocks $M$, and double tolerance parameter $\varepsilon$.}
\Output{Matrices $\mtx{Q} \in \mathbb{R}^{m \times (b i)}$, $\mtx{B} \in \mathbb{R}^{(b i) \times n}$,
s.t. $\|\mtx{Q} \mtx{B} - \mtx{A}\| < \varepsilon$ (if $M$ large enough).}
\BlankLine
\For{$i=1$ \KwTo $M$}{
    $\mtx{\Omega}_{i} = \texttt{randn}(n,b)$\;
    $\mtx{Q}_{i} = \texttt{orth}(\mtx{A}\mtx{\Omega}_{i})$\;
    \For{$j=1$ \KwTo $q$}{
        $\mtx{Q}_{i} = \texttt{orth}(\mtx{A}^{*}\mtx{Q}_{i})$ \;
        $\mtx{Q}_{i} = \texttt{orth}(\mtx{A}    \mtx{Q}_{i})$ \;
    }
    $\mtx{Q}_{i} = \texttt{orth}(\mtx{Q}_{i} - \sum_{j=1}^{i-1}\mtx{Q}_{j}\mtx{Q}_{j}^{*}\mtx{Q}_{i})$\;
    $\mtx{B}_{i} = \vct{Q}_{i}^{*}\mtx{A}$\;
    $\mtx{A} = \mtx{A} - \mtx{Q}_{i}\mtx{B}_{i}$\;
    \textbf{if} $\|\mtx{A}\| < \varepsilon$ \textbf{then break}\;
}
$\mtx{Q} = [\mtx{Q}_{1}\ \cdots\ \mtx{Q}_{i}]$ and $\mtx{B} = [\mtx{B}_{1}^{*}\ \cdots\ \mtx{B}_{i}^{*}]^{*}$\;
\end{algorithm}
\end{figure}

\begin{figure}
\begin{algorithm}[H]
\SetKwInOut{Input}{Input}
\SetKwInOut{Output}{Output}
\caption{An approximate parallelizable randomized blocked \QB~decomposition \label{algo:randpbQB3}}
\Input{$\mtx{A}\in\mathbb{R}^{m\times n}$, integer block size parameter $b$,
maximum number of blocks $M$.}
\Output{Matrices $\mtx{Q} \in \mathbb{R}^{m \times (b M)}$, $\mtx{B} \in \mathbb{R}^{(b M) \times n}$,
s.t. $\mtx{Q} \mtx{B} \approx \mtx{A}$.}
\BlankLine
$\mtx{Q} = \left[\right]$\;
\For{$i=1$ \KwTo $M$}{
    $\mtx{\Omega}_{i} = \texttt{randn}(n,b)$\;
    $\mtx{Y}_{i} = \mtx{A} \mtx{\Omega}_{i}$\;
    \For{$j=1$ \KwTo $q$}{
        $\mtx{Q}_{i} = \texttt{orth}(\mtx{Y}_{i})$\;
        $\mtx{Y}_{i} = \mtx{A}^* \mtx{Q}_{i}$\;
        $\mtx{Q}_{i} = \texttt{orth}(\mtx{Y}_{i})$\;
        $\mtx{Y}_{i} = \mtx{A} \mtx{Q}_{i}$\;
    }
}

\For{$i=1$ \KwTo $M$}{
    $\mtx{Q}_{i} = \orth(\mtx{Y}_{i})$\;
}

\For{$i=1$ \KwTo $M$}{
    $\mtx{Q}_{i} = \mtx{Q}_i - \mtx{Q} \mtx{Q}^* \mtx{Q}_i$\;
    $\mtx{Q}_{i} = \orth(\mtx{Q}_{i})$\;
    $\mtx{Q} = \left[\mtx{Q}, \mtx{Q}_i\right]$\;
}

$\mtx{B} = \mtx{Q}^* \mtx{A}$\;
\end{algorithm}
\end{figure}

\end{appendix}

\end{document}